\documentclass[a4paper, 11pt]{article}

\usepackage{indentfirst}
\usepackage{amsmath}
\usepackage{amsthm}
\usepackage{amsfonts}
\usepackage{bbm}
\usepackage{accents}
\usepackage{xcolor}

\newtheorem{thm}{Theorem}
\newtheorem{remark}{Remark}
\newtheorem{lemma}{Lemma}
\newtheorem{cor}{Corollary}
\newtheorem{proposition}{Proposition}

\theoremstyle{definition}

\newcommand{\eee}{{\rm e}}

\newcommand{\dod}{\overset{{\rm d}}{\to}}
\newcommand{\me}{\mathbb{E}}
\newcommand{\mn}{\mathbb{N}}

\newcommand{\mmp}{\mathbb{P}}
\def\1{\mathbbm{1}}

\begin{document}

\title{Limit theorems for random Dirichlet series with summation over primes, with an application to Rademacher random multiplicative functions}\date{}
\author{Congzao Dong\footnote{School of Mathematics and Statistics, Xidian University, Xi’an, China; e-mail address: czdong@xidian.edu.cn} \ \ and \ \ Alexander Iksanov\footnote{Faculty of Computer Science and Cybernetics, Taras Shevchenko National University of Kyiv, Ukraine; e-mail address:
iksan@univ.kiev.ua}}
\maketitle
\begin{abstract}
\noindent It is shown that two conjectures put forward in the recent article Iksanov and Kostohryz (2025) are true. Namely, we prove a functional central limit theorem (FCLT) and a law of the iterated logarithm (LIL) for a random Dirichlet series $\sum_p \frac{\eta_p}{p^{1/2+s}}$ as $s\to 0+$, where $\eta_1$, $\eta_2,\ldots$ are independent identically distributed random variables with zero mean and finite variance, and $\sum_p$ denotes the summation over the prime numbers. As a consequence, we obtain an FCLT and an LIL for $\log \sum_{n\geq 1} \frac{f(n)}{n^{1/2+s}}$ as $s\to 0+$, where $f$ is a Rademacher random multiplicative function.
\end{abstract}

\noindent Key words: functional limit theorem; law of the iterated logarithm; Rademacher random multiplicative function; random Dirichlet series

\noindent 2000 Mathematics Subject Classification: Primary: 60F15, 60F17 \\
\hphantom{2000 Mathematics Subject Classification: } Secondary: 11N37; 60G50

\section{Introduction and main results}

Let $\eta_1$, $\eta_2,\ldots$ be independent copies of a random variable $\eta$ with zero mean and finite variance, which live on a probability space $(\Omega, \mathfrak{F}, \mmp)$. We present in this article limit theorems for a random Dirichlet series $X(s):=\sum_p \frac{\eta_p}{p^{1/2+s}}$ as $s\to 0+$. Here and hereafter, $\sum_p$ denotes summation over the set $\mathcal{P}$ of prime numbers. By Kolmogorov's three series theorem, for each $s>0$, the series defining $X(s)$ converges almost surely (a.s.) and absolutely diverges a.s.

Let $\alpha\geq -1/2$. Functional central limit theorems (FCLTs) and laws of the iterated logarithm (LILs) for a random Dirichlet series $\sum_{k\geq 2}\frac{(\log k)^\alpha}{k^{1/2+s}}\eta_k$ have attracted some attention in the recent past, see \cite{Aymone+Frometa+Misturini:2020} for an LIL in the case $\alpha=0$ and $\mmp\{\eta=\pm 1\}=1/2$, \cite{Buraczewskietal:2023} and \cite{Iksanov+Kostohryz:2025} for an FCLT and an LIL in the cases $\alpha>-1/2$ and $\alpha=-1/2$, respectively. Theorem 4.1 in \cite{Zhao+Huang:2024} is an LIL-like result under the assumption that the distribution of $\eta$ is symmetric $\gamma$-stable for $\gamma\in (0,2]$.

Here are our main results. We start with an FCLT on $C[0,\infty)$ the space of real-valued continuous functions defined on $[0,\infty)$. It is assumed that $C[0,\infty)$ is equipped with the topology of locally uniform convergence. As usual, $\Longrightarrow$ denotes weak convergence in a function space.
\begin{thm}\label{con:clt}
Assume that $\me[\eta]=0$ and $\sigma^2=\me [\eta^2]\in (0,\infty)$. Then $$\Big(\frac{1}{(\log 1/s)^{1/2}}\sum_p\frac{\eta_p}{p^{1/2+s^t}}\Big)_{t\geq 0}~\Longrightarrow~ (\sigma B(t))_{t\geq 0},\quad s\to 0+$$ on $C[0,\infty)$, where $(B(t))_{t\geq 0}$ is a standard Brownian motion.
\end{thm}

Paul Bourgade has kindly informed us that Theorem 1.1 in \cite{Vassaux:2025} proves that $$\Big(\frac{1}{(\log\log T)^{1/2}}\log\zeta(1/2+(\log T)^{-t}+{\rm i}\tau_T)\Big)_{t\in [0,1]}~\Longrightarrow~ (B^\ast(t))_{t\in [0,1]},\quad T\to\infty$$ on the space of complex-valued continuous functions defined on $[0,1]$, where $\tau_T$ is a random variable with the (continuous) uniform distribution on $[0,T]$, $\zeta$ is the Riemann zeta function, and $(B^\ast(t))_{t\in [0,1]}$ is a standard complex Brownian motion. To see a link with our Theorem \ref{con:clt}, observe that replacing $s$ with $1/\log T$ the result of Theorem \ref{con:clt} can be written as $$\Big(\frac{1}{(\log\log T)^{1/2}}\sum_p\frac{\eta_p}{p^{1/2+(\log T)^{-t}}}\Big)_{t\geq 0}~\Longrightarrow~ (\sigma B(t))_{t\geq 0},\quad T\to\infty$$ on $C[0,\infty)$. Further, finite-dimensional distributions of $\big(\log\zeta(1/2+(\log T)^{-t}+{\rm i}\tau_T)\big)_{t\in [0,1]}$ are well approximated by those of $\big(\sum_{p\leq T}p^{-1/2-(\log T)^{-t}-{\rm i}\tau_T})\big)_{t\in [0,1]}$, see p.~5 in \cite{Vassaux:2025}. It would be very interesting to understand whether this resemblance is merely coincidental or reflects a genuine connection. 

We proceed with an LIL. For a family $(x_s)$ of real numbers denote by $C((x_s))$ the set of its limit points as $s\to 0+$.
\begin{thm}\label{con:lil}
Assume that $\me[\eta]=0$ and $\sigma^2=\me[\eta^2]\in(0,\infty)$. Then
\begin{equation}\label{eq:limitpoints}
C\Big(\Big(\Big(\frac{1}{2\sigma^2\log 1/s\, \log \log\log 1/s}\Big)^{1/2}\sum_p\frac{\eta_p}{p^{1/2+s}}: s\in(0, \eee^{-\eee})\Big)\Big)=[-1,1] \quad\text{\rm a.s.}
\end{equation}
In particular,
\begin{equation}\label{eq:limsup}
\limsup_{s\to 0+}\Big(\frac{1}{\log 1/s\, \log\log\log 1/s}\Big)^{1/2}\sum_p\frac{\eta_p}{p^{1/2+s}}=\sqrt{2}\sigma 
\quad\text{\rm a.s.}
\end{equation}
and
\begin{equation}\label{eq:liminf}
\liminf_{s\to 0+}\Big(\frac{1}{\log 1/s\, \log\log\log 1/s}\Big)^{1/2}\sum_p\frac{\eta_p }{p^{1/2+s}}=-\sqrt{2}\sigma 
\quad\text{\rm a.s.}
\end{equation}
\end{thm}

Although a triple logarithm appears in formulae \eqref{eq:limsup} and \eqref{eq:liminf}, Theorem \ref{con:lil} is indeed an LIL. According to relation \eqref{eq:var} given below, $\me [(X(s))^2]\sim \sigma^2 \log 1/s$ as $s\to 0+$, whence $$\log\log \me [(X(s))^2]~\sim~ \log\log\log 1/s,\quad s\to 0+.$$ Here, $y(s)\sim z(s)$ as $s\to 0+$ means that $\lim_{s\to 0+}(y(s)/z(s))=1$.

Theorem \ref{con:lil} significantly improves upon 
Lemma 2.1 in \cite{Aymone+Heap+Zhao:2023} and Theorem 1.3 in \cite{Geis+Hiary:2025}.

Theorems \ref{con:clt} and \ref{con:lil} were given as Conjectures 1 and 2 in the recent article \cite{Iksanov+Kostohryz:2025}. It was shown in the cited paper that the aforementioned FCLT and LIL hold true with
$\sum_{k\geq 2}\frac{\eta_k}{(\log k)^{1/2} k^{1/2+s}}$ replacing $\sum_p\frac{\eta_p }{p^{1/2+s}}$. Actually, Conjectures 1 and 2 in \cite{Iksanov+Kostohryz:2025} were due to a comment of one of the anonymous referees of that paper. Neither the authors of \cite{Iksanov+Kostohryz:2025} nor the present authors know 
how to {\it derive} the limit theorems for $\sum_p\frac{\eta_p }{p^{1/2+s}}$ from the corresponding results for $\sum_{k\geq 2}\frac{\eta_k}{(\log k)^{1/2} k^{1/2+s}}$. Let $p_1<p_2<\ldots$ be the prime numbers arranged in the order of increase. The prime number theorem ensures that $p_k\sim k\log k$ as $k\to\infty$. In view of this, one expects intuitively that $$\sum_k \frac{\eta_{p_k}}{p_k^{1/2+s}}\asymp \sum_k \frac{\eta_{p_k}}{(k\log k)^{1/2+s}}\asymp \sum_k \frac{\eta_k}{(\log k)^{1/2} k^{1/2+s}},\quad s\to 0+,$$ where $\asymp$ denotes some form of asymptotic closeness, without precise meaning. Not being able to make these asymptotic relations precise, we prove Theorems \ref{con:clt} and \ref{con:lil} by following the line of reasoning worked out in \cite{Iksanov+Kostohryz:2025}. Naturally, some technical details differ. 
At this stage, 
we only note that $\#\{k\in\mn: k\leq x\}\sim x$ as $x\to\infty$, whereas, by the prime number theorem, $\#\{p\in\mathcal{P}: p\leq x\}\sim x/\log x$.

Next, we discuss an application to Probabilistic Number Theory. To match the notation used in number-theoretic articles, we write $f(p)$ in place of $\eta_p$ if $\mmp\{\eta_p=\pm 1\}=1/2$ for $p\in\mathcal{P}$. Wintner in \cite{Wintner:1944} introduced a Rademacher random multiplicative function $f$ as a model for the M\"{o}bius function. The function is defined by $f(1):=1$, $f(n):=\prod_{p|n} f(p)$ for square-free $n$ (that is, those given by products of distinct prime numbers) and $f(n):=0$ for $n$, which are not square-free. Thus, for instance, $f(6)=f(2)f(3)$, $f(12)=0$ and $f(30)=f(2)f(3)f(5)$. An important open problem is to prove an LIL (or an appropriate counterpart) for $\sum_{n\leq x}f(n)$ as $x\to\infty$. As far as we know, the best one-sided results available at the moment are $$\Big|\sum_{n\leq x}f(n)\Big|=O(x^{1/2}(\log\log x)^{3/4+\varepsilon}),\quad x\to\infty\quad\text{a.s.~for all}~\varepsilon>0,$$ see Theorem 1.1 in \cite{Caich:2025+}, and that there a.s.\ exist arbitrarily large values of $x$ for which $$\Big|\sum_{n\leq x}f(n)\Big|\geq x^{1/2}(\log\log x)^{1/4+o(1)},$$ see Theorem 1 in \cite{Harper:2023}. The two aforementioned papers provide surveys of earlier works, in which weaker results have been proved. Distributional convergence of $\sum_{n\leq x}f(n)$, properly normalized, as $x\to\infty$ does not seem to have been proved either, even in the sense of one-dimensional convergence. It is shown on p.~99 in \cite{Harper:2013} that the distributional limit of $(\me [(\sum_{n\leq x}f(n))^2])^{-1/2} \sum_{n\leq x}f(n)$ is not Gaussian.

Recall that a Steinhaus random multiplicative function $\alpha$ is defined as follows. The variables $(\alpha(p))_{p\in\mathcal{P}}$ are independent and uniformly distributed on $\{z\in\mathbb{C}: |z|=1\}$, where $\mathbb{C}$ denotes the set of complex numbers. Further, if $n=\prod_p p^{b_p}$, then $\alpha(n)=\prod_p (\alpha(p))^{b_p}$. It was very recently shown in \cite{Gorodetsky+Wong:2025}, see also \cite{Hardy:2025} for an important partial result, that $x^{-1/2}(\log\log x)^{1/4}\sum_{n\leq x}\alpha(n)$ converges in distribution to a random multiple of a random variable with the standard complex normal distribution.

We do not contribute directly to the open problems. By an application of Theorems \ref{con:clt} and \ref{con:lil}, we obtain an FCLT and an LIL for the logarithm of a random Dirichlet series with weights $f(n)$. It is known that $F$ defined by $F(s):=\sum_{n\geq 1} f(n)n^{-1/2-s}$ for $s>0$ admits an Euler product representation
\begin{equation}\label{eq:Euler}
F(s)=\prod_p \Big(1+\frac{f(p)}{p^{1/2+s}}\Big),\quad s>0.
\end{equation}
Since the random series $\sum_p f(p)p^{-1/2-s}$ converges a.s. and the series $\sum_p p^{-1-2s}$ converges, this representation particularly ensures that the series defining $F(s)$ converges a.s. The product on the right-hand side of the last equality is strictly positive for each $s>0$. Hence, the function $s\mapsto \log F(s)$ is real-valued.
\begin{cor}\label{cor:clt}
$$\Big(\frac{1}{(\log 1/s)^{1/2}}\Big(\log \sum_{n\geq 1} \frac{f(n)}{n^{1/2+s^t}}+\frac{\log \zeta(1+2s^t)}{2} \Big)\Big)_{t\geq 0}~\Longrightarrow~ (B(t))_{t\geq 0},\quad s\to 0+$$ on $C[0,\infty)$, where $(B(t))_{t\geq 0}$ is a standard Brownian motion, and $\zeta(r):=\sum_{n\geq 1}n^{-r}$ for $r>1$.
\end{cor}

\begin{cor}\label{cor:lil}
\begin{equation*}
C\Big(\Big(\Big(\frac{1}{2\log 1/s\, \log \log\log 1/s}\Big)^{1/2}\Big(\log \sum_{n\geq 1}\frac{f(n)}{n^{1/2+s}}+\frac{\log \zeta(1+2s)}{2}\Big): s\in(0, \eee^{-\eee})\Big)\Big)=[-1,1] \quad\text{\rm a.s.}
\end{equation*}
\end{cor}

Recall that, for $k\geq 2$, a positive integer number is called $k$-free if it has no divisor which is a perfect $k$th power. In particular, a $2$-free number is a square-free number. Define now a random multiplicative function $f$ by $f(1):=1$, $f(n):=\prod_{p^j|n} (f(p))^j$ for $k$-free $n$, $j<k$ and $f(n):=0$ for $n$, which are not $k$-free. For instance, if $k=3$, then $f(12)=f(3)(f(2))^2=f(3)$ and $f(24)=0$. Our last result demonstrates that, for $k\geq 3$, both FCLT and LIL for $\log \sum_{n\geq 1} f(n)n^{-1/2-s}$ as $s\to 0+$ are universal and take the forms which are slightly different from those in the case $k=2$.
\begin{cor}\label{cor:cltlil2}
Let $k\geq 3$. Then the statements of Corollaries \ref{cor:clt} and \ref{cor:lil} hold true with `$-$' replacing `$+$'. For instance,
$$\Big(\frac{1}{(\log 1/s)^{1/2}}\Big(\log \sum_{n\geq 1} \frac{f(n)}{n^{1/2+s^t}}-\frac{\log \zeta(1+2s^t)}{2} \Big)\Big)_{t\geq 0}~\Longrightarrow~ (B(t))_{t\geq 0},\quad s\to 0+$$ on $C[0,\infty)$.
\end{cor}
\begin{remark}
While Corollaries \ref{cor:clt} and \ref{cor:lil} are results on fluctuations of $\sum_{n\geq 1}\frac{f(n)}{n^{1/2+s}}(\sum_{k\geq 1}\frac{1}{k^{1+2s}})^{1/2}$ as $s\to 0+$, Corollary \ref{cor:cltlil2} is concerned with fluctuations of $\sum_{n\geq 1}\frac{f(n)}{n^{1/2+s}}(\sum_{k\geq 1}\frac{1}{k^{1+2s}})^{-1/2}$.
\end{remark}

Sums of independent random variables indexed by prime numbers and random Euler products like $F$ in \eqref{eq:Euler} pop up frequently in the number-theoretic literature, see, for instance, Section 5 in \cite{Klurman+Lamzouri+Munsch:2025+} and already cited Lemma 2.1 in \cite{Aymone+Heap+Zhao:2023} and Theorem 1.3 in \cite{Geis+Hiary:2025}, and Section 3 in \cite{Granville+Soundararajan:2003}. To be more specific, we only mention that Theorem 3.1 in \cite{Granville+Soundararajan:2003} provides a precise asymptotic behavior with remainder of $\mmp\{\prod_{p\leq y}(1-g(p)/p)^{-1}>x\}$ and $\mmp\{\prod_{p\leq y}(1-g(p)/p)^{-1}\leq 1/x\}$ as $x$ and $y$ tend to $\infty$. Here, for $p\in\mathcal{P}$, the $g(p)$ are independent random variables with $\mmp\{g(p)=\pm 1\}=p/(2(p+1))$ and $\mmp\{g(p)=0\}=1/(p+1)$.

The remainder of the paper is organized as follows. We prove Theorems \ref{con:lil} and \ref{con:clt} in Sections \ref{sect:aux} and \ref{sect:flt}, respectively. 
Our proof of Theorem \ref{con:lil} is necessarily more technical than that of Theorem \ref{con:clt}. On the other hand, a chaining argument used in the proof of Theorem \ref{con:clt} rests heavily upon the reasoning given in the proof of Theorem \ref{con:lil}. Summarizing, it is more convenient to prove Theorem \ref{con:lil} first. We explain in Section \ref{sect:cor} how to derive Corollaries \ref{cor:clt} and \ref{cor:lil} from Theorems  \ref{con:clt} and  \ref{con:lil}. Finally, a comment on the proof of Corollary \ref{cor:cltlil2} is given in Section \ref{sect:cor3}.

\section{Proof of Theorem \ref{con:lil}}\label{sect:aux}
Put $$g(s):=\mathbb{E}\Big[\Big(\sum_{p}\frac{\eta_p}{p^{1/2+s}}\Big)^2\Big]=\sigma^2 \sum_p \frac{1}{p^{1+2s}},\quad s>0.$$ By Proposition A.3 in \cite{Geis+Hiary:2025},
\begin{equation}\label{eq:var}
g(s)~ \sim~ \sigma^2 \log(1/s),\quad s\to 0+.
\end{equation}

We obtain separately in Propositions \ref{pr1} and \ref{pr2} the upper bound for $\limsup$ and the lower bound for $\limsup$, respectively. The reason is that the arguments beyond these two results are essentially different.
From now on, we write $\log^{(3)}$ for $\log\log\log$.
\begin{proposition}\label{pr1}
Assume that $\me[\eta]=0$ and $\sigma^2=\me[\eta^2]\in(0,\infty)$. Then
\begin{equation}\label{5.1}
\limsup_{s\to 0+}\Big(\frac{1}{\log 1/s\, \log^{(3)} 1/s}\Big)^{1/2}\sum_p \frac{\eta_p}{p^{1/2+s}}\leq \sqrt{2}\sigma 
\quad\text{\rm a.s.}
\end{equation}
and
\begin{equation}\label{5.2}
\liminf_{s\to 0+}\Big(\frac{1}{\log 1/s\, \log^{(3)}1/s}\Big)^{1/2}\sum_p \frac{\eta_p}{p^{1/2+s}}\geq-\sqrt{2}\sigma 
\quad\text{\rm a.s.}
\end{equation}
\end{proposition}

Now we explain our strategy of proving \eqref{5.1}. We split the sum $\sum_p \frac{\eta_p}{p^{1/2+s}}$ into two fragments: initial and final. It is shown in Lemma \ref{lemma_2} that the contribution of the initial fragment vanishes. We prove in Lemma \ref{lemma_3} that the contribution of the final fragment vanishes, too if the variables $\eta_p$ are properly truncated. Lemma \ref{lemma4} is concerned with the final (principal) fragment of the series, in which the variables $\eta_p$ are differently truncated and also centered. Here, it is proved that the desired convergence only holds along a sequence. The most involved preparatory result is Lemma \ref{lemma5}. It shows, via a chaining argument, that the convergence along a sequence obtained in Lemma \ref{lemma4} can be upgraded to the convergence along the real numbers. We note in passing that a complicating factor here is that the finiteness of exponential moments of $\eta$ is not assumed. Under the assumption $\me [\eee^{s\eta}]<\infty$ for all $s>0$, a proof would have been much easier and shorter.

As has already been stated, our proof of Theorem \ref{con:lil}, and particularly of Proposition \ref{pr1}, follows the path of the proof of Theorem 2 in \cite{Iksanov+Kostohryz:2025}. The adaptation runs, for the most part, smoothly. The only exception is that an extra effort is required for proving (the most troublesome) Lemma \ref{lemma5}.

For $s\in(0, \eee^{-\eee})$, put $$L(s)=\Big(\frac{1}{2\log 1/s\, \log^{(3)}1/s}\Big)^{1/2}.$$
Let $M: (0,\infty)\to \mn$ denote a function satisfying $\lim_{s\to 0+}M(s)=+\infty$ and
\begin{equation}\label{eq:growth}
\lim_{s\to 0+}\frac{M(s)}{\log 1/s}=0.
\end{equation}
Replacing $\eta_p$ with $\eta_p/\sigma$ we can work under the assumption that $\sigma^2=1$. Thus, in what follows, it is tacitly assumed that $\me[\eta]=0$ and $\me[\eta^2]=1$.
\begin{lemma}\label{lemma_2}
The following convergence 
holds
$$\lim_{s\to0+} L(s)\sum_{p\leq M(s)} \frac{\eta_p}{p^{1/2+s}}=0\quad\text{\rm a.s.}$$
\end{lemma}
\begin{proof}
Let $p_1<p_2<\ldots$ be the rearrangement of the prime numbers in the increasing order. Put $T_0:=0$ and $T_n:=\eta_{p_1}+\ldots+\eta_{p_n}$ for $n\in\mn$, and then $T^\ast(x):=\sum_{p\leq x}\eta_p$ for $x\geq 1$. Observe that $T^\ast(x)=T_{\rho(x)}$ for $x\geq 1$, where $\rho(x):=\#\{p\in\mathcal{P}: p\leq x\}$.

A functional LIL for standard random walks (see, for instance, Corollary 5.3.5 on p.~294 in \cite{Stout:1974}) entails
\begin{equation*}
\max_{k\leq n}\,|T_k|=O\big((n\log\log n)^{1/2}\big),\quad n\to\infty \quad\text{a.s.}
\end{equation*}
According to the prime number theorem, $\rho(x)\sim x/\log x$ as $x\to\infty$. A combination of these two facts yields
\begin{equation}\label{5.3}
|T^\ast(x)|=|T_{\rho(x)}|\leq \max_{1\leq k\leq \rho(x)}\,|T_k|=O\big((x\log\log x/\log x)^{1/2}\big),\quad x\to\infty \quad\text{a.s.}
\end{equation}
Integrating by parts we infer
\begin{equation*}
\sum_{p\leq M(s)} \frac{\eta_p}{p^{1/2+s}}=\int_{(1,\,M(s)]}\frac{{\rm d}T^\ast(x)}{x^{1/2+s}}=\frac{T^\ast(M(s))}{(M(s))^{1/2+s}}\\+(1/2+s)\int_{1}^{M(s)}\frac{T^\ast(x)}{x^{3/2+s}}{\rm d}x.
\end{equation*}
Relation \eqref{eq:growth} entails
$\lim_{s\to 0+}(M(s))^s=1$. This together with \eqref{5.3} enables us to conclude that,
as $s\to 0+$,
\begin{equation*}
\frac{|T^\ast(M(s))|}{(M(s))^{1/2+s}}~\sim~ \frac{|T^\ast(M(s))|}{(M(s))^{1/2}}=O((\log M(s))^{-1/2}(\log\log M(s))^{1/2}) = o(1).
\end{equation*}
Since $\lim_{s\to 0+}L(s)=0$, the latter ensures that
\begin{equation*}
\lim_{s\to 0+} L(s)\frac{|T^\ast(M(s))|}{(M(s))^{1/2+s}}= 0 \quad\text{a.s.}
\end{equation*}
We are left with showing that $$\lim_{s\to 0+}L(s)\int_{1}^{M(s)}\frac{|T^\ast(x)|}{x^{3/2+s}}{\rm d}x=0\quad\text{a.s.}$$ To this end, write, with the help of \eqref{5.3},
\begin{multline*}
\int_{1}^{M(s)}\frac{|T^\ast(x)|}{x^{3/2+s}}{\rm d}x\leq \big(\max_{y\leq M(s)}\,|T^\ast(y)|\big) \int_{1}^{M(s)}\frac{{\rm d}x}{x^{3/2+s}}=\big(\max_{k\leq \rho(M(s))}\,|T_k|\big)O(1)\\=O((M(s)\log\log M(s)/\log M(s))^{1/2}),\quad s\to 0+\quad\text{a.s.}
\end{multline*}
Finally, \eqref{eq:growth} entails $$\lim_{s\to 0+}\frac{M(s)\log\log M(s)}{\log 1/s\,\log^{(3)}1/s}=0.$$ The proof of Lemma \ref{lemma_2} is complete.
\end{proof}

For $p\in \mathcal{P}$, $\theta>0$ and $s\in(0, \eee^{-\eee})$, define the event
\begin{equation*}
A_{p,\theta}(s):=\Big\{|\eta_p|>\frac{\theta}{\log\log 1/s}\Big(\frac{p^{1+2s}\, g(s)}{\log^{(3)}1/s}\Big)^{1/2}\Big\}.
\end{equation*}

\begin{lemma}\label{lemma_3}
For all $\theta>0$,
\begin{equation}\label{lem4.1}
\lim_{s\to 0+}\sum_{p\ge M(s)+1}\frac{|\eta_p|\1_{A_{p,\theta}(s)}}{p^{1/2+s}}=0\quad \text{\rm{a.s.}}
\end{equation}
and
\begin{equation}\label{lem4.2}
\lim_{s\to 0+} \sum_{p\ge M(s)+1}\frac{\me\big[|\eta_p|\1_{A_{p,\theta}(s)}\big]}{p^{1/2+s}}=0.
    \end{equation}
\end{lemma}
\begin{proof}
Put $h(s):=(\log\log 1/s)(\log^{(3)}1/s)^{1/2}$. For $s\in (0, \eee^{-\eee})$,
\begin{equation*}
\sum_{p\ge M(s)+1}\frac{|\eta_p|\1_{A_{p,\theta}(s)} }{p^{1/2+s}} \le \sum_{p\ge M(s)+1}\frac{|\eta_p|}{p^{1/2}}\1_{\{p^{-1/2}|\eta_p|>\theta(g(s))^{1/2} (h(s))^{-1}\}}\quad\text {{\rm a.s.}}
\end{equation*}
The assumption $\mathbb{E}[\eta^2]<\infty$ entails
\begin{equation}\label{eq:zero}
\lim_{k\to\infty}k^{-1/2}|\eta_k|=0\quad\text{a.s.}
\end{equation}
and thereupon $\sup_{k\ge 1}(k^{-1/2}|\eta_k|)<\infty$ a.s. Since $\lim_{s\to 0+}((g(s))^{1/2} (h(s))^{-1})=+\infty$, we infer
\begin{equation*}
\1_{\{p^{-1/2}|\eta_p|>\theta (g(s))^{1/2} (h(s))^{-1}\}}\le \1_{\{\sup_{k\ge 1}\,(k^{-1/2}|\eta_k|)>\theta (g(s))^{1/2} (h(s))^{-1}\}}=0
\end{equation*}
a.s.\ for small $s$. We have proved that the sum in \eqref{lem4.1} is equal to $0$ a.s.\ for small enough $s$.

Relation \eqref{lem4.2} is justified as follows:
\begin{multline*}
\sum_{p\ge M(s)+1}\frac{\me \big[|\eta_p|\1_{A_{p,\theta}(s)}\big]}{p^{1/2+s}}\le \sum_{p\ge M(s)+1}p^{-1/2}\me\big[|\eta|\1_{\{\theta^{-1}(g(s))^{-1/2} h(s)|\eta|>p^{1/2}\}}\big]\\ \le \me \Bigg[|\eta|\sum_{p\leq \lfloor \theta^{-2}(g(s))^{-1}(h(s))^2\eta^2\rfloor}p^{-1/2}\Bigg]  \le \me \Bigg[|\eta|\sum_{k\leq \lfloor \theta^{-2}(g(s))^{-1}(h(s))^2\eta^2\rfloor}k^{-1/2}\Bigg]\\\leq 2\theta^{-1}\me[\eta^2](g(s))^{-1/2} h(s)=2\theta^{-1}(g(s))^{-1/2}h(s)~\to~0,\quad s\to 0+.
\end{multline*}
The proof of Lemma \ref{lemma_3} is complete.
\end{proof}
In what follows, $(A_{p,\theta}(s))^c$ denotes the complement of $A_{p,\theta}(s)$, that is, for $p\in\mathcal{P}$, $\theta>0$ and $s\in(0, \eee^{-\eee})$,
\begin{equation*}
(A_{p,\theta}(s))^c:=\Big\{|\eta_p|\le\frac{\theta}{\log\log1/s}\Big(\frac{p^{1+2s}\, g(s)}{\log^{(3)}1/s}\Big)^{1/2}\Big\}.
\end{equation*}

\begin{lemma}\label{lemma4}
Fix any $\gamma\in (0,(\sqrt{5}-1)/2)$, pick any $\theta=\theta(\gamma)$ satisfying
\begin{equation}\label{5.7}
(1-\gamma)(1+\gamma)^2(2-\exp(2\sqrt{2}(1+\gamma)\theta))>1
\end{equation}
and put $s_n:=\exp(-\exp(n^{1-\gamma}))$ for $n\in\mathbb{N}$. Then
\begin{equation*}
\limsup_{n\to\infty} L(s_n)\sum_{p\ge M(s_n)+1}\frac{\widetilde{\eta}_{p,\theta}(s_n)}{p^{1/2+s_n}}\le 1+\gamma \quad\text{{\rm a.s.},}
\end{equation*}
where $\widetilde{\eta}_{p,\theta}(s):=\eta_p \1_{(A_{p,\theta}(s))^c}-\me \big[\eta_p\1_{(A_{p,\theta}(s))^c}\big]$ for $p\in\mathcal{P}$ and $s\in(0, \eee^{-\eee})$.
\end{lemma}
\begin{proof}
Since $(1-\gamma)(1+\gamma)^2>1$ whenever $\gamma\in (0, (\sqrt{5}-1)/2)$, $\theta$ satisfying \eqref{5.7} does indeed exist.

Put $L^\ast(s):=(2g(s)\,\log^{(3)}1/s)^{-1/2}$ for $s\in (0, \eee^{-\eee})$. Since $L^\ast(s)\sim L(s)$ as $s\to 0+$, we can and do prove the result, with $L^\ast$ replacing $L$. Put
\begin{equation*}
X(s)=L^\ast(s)\sum_{p\ge M(s)+1}\frac{\widetilde{\eta}_{p,\theta}(s)}{p^{1/2+s}},\quad s\in(0, \eee^{-\eee}).
\end{equation*}
Using $\eee^x\le 1+x+(x^2/2)\eee^{|x|}$ for $x\in\mathbb{R}$ and $\mathbb{E}[\widetilde{\eta}_{p,\theta}(s)]=0$ we deduce, for $u\in\mathbb{R}$,
\begin{multline*}
\mathbb{E}[\eee^{uX(s)}]=\prod_{p\ge M(s)+1}\mathbb{E}\Big[\exp\Big(u L^\ast(s)\frac{\widetilde{\eta}_{p,\theta}(s)}{p^{1/2+s}}\Big)\Big] \\ \le \prod_{p\ge M(s)+1}\Big(1+\frac{u^2(L^\ast(s))^2}{2}\frac{1}{p^{1+2s}}\me\Big[(\widetilde{\eta}_{p,\theta}(s))^2\exp\Big(|u|L^\ast(s)
\frac{|\widetilde{\eta}_{p,\theta}(s)|}{p^{1/2+s}}\Big)\Big]\Big).
\end{multline*}
The inequality
\begin{multline}\label{5.8}
|\widetilde{\eta}_{p,\theta}(s)|\le|\eta_p|\1_{(A_{p,\theta}(s))^c}+\mathbb{E}\big[|\eta_p|\1_{(A_{p,\theta}(s))^c}\big]\\\le \frac{2\theta p^{1/2+s}}{\log\log 1/s}\Big(\frac{g(s)}{\log^{(3)}1/s}\Big)^{1/2}\le 2\theta p^{1/2+s} \Big(\frac{g(s)}{\log^{(3)}1/s}\Big)^{1/2}\quad\text{a.s.,}
\end{multline}
which is valid for $p\in\mathcal{P}$ and $s\in(0, \eee^{-\eee})$, implies that
\begin{equation*}
\exp\Big(|u|L^\ast(s)\frac{|\widetilde{\eta}_{p,\theta}(s)|}{p^{1/2+s}}\Big)\le\exp\Big(\frac{\sqrt{2}\theta|u|}{\log^{(3)}1/s}\Big)\quad\text{a.s.}
\end{equation*}
Together with the inequalities $\me[(\widetilde{\eta}_{p,\theta}(s))^2]\le 1$ and $1+x\le \eee^x$ for $x\in\mathbb{R}$ this gives,
for $u\in\mathbb{R}$,
\begin{multline}\label{5.9}
\mathbb{E}[\eee^{uX(s)}]\le\prod_{p\ge M(s)+1}\exp\Big(\frac{u^2(L^\ast(s))^2}{2}\frac{1}{p^{1+2s}}\exp\Big(\frac{\sqrt{2}\theta|u|}{\log^{(3)}1/s}\Big)\Big)\\
\le\exp\Big(\frac{u^2}{4 \log^{(3)}1/s}\exp\Big(\frac{\sqrt{2}\theta |u|}{\log^{(3)}1/s}\Big)\Big).
\end{multline}
An application of Markov's inequality yields, for $u\ge 0$,
\begin{multline*}
\mathbb{P}\{X(s_n)>1+\gamma\}\le \eee^{-(1+\gamma)u} \mathbb{E}[\eee^{u X(s_n)}] \\ \le \exp\Big(-(1+\gamma)u+\frac{u^2}{4 \log^{(3)}1/s_n}\exp\Big(\frac{\sqrt{2}\theta u}{ \log^{(3)}1/s_n}\Big)\Big).
\end{multline*}
Putting $u=2(1+\gamma)\log^{(3)}1/s_n$ we conclude that
\begin{multline*}
\mathbb{P}\{X(s_n)>1+\gamma\}\le \exp(-(1+\gamma)^2(2-\exp(2\sqrt{2}(1+\gamma)\theta)) \log^{(3)} 1/s_n)\\ =\frac{1}{n^{(1-\gamma)(1+\gamma)^2(2-\exp(2\sqrt{2}(1+\gamma)\theta))}}.
\end{multline*}
Thus, in view of \eqref{5.7}, $\sum_{n\ge 1}\mathbb{P}\{X(s_n)>1+\gamma\}<\infty$, and invoking the direct part of the Borel-Cantelli lemma completes the proof of Lemma \ref{lemma4}.
\end{proof}

\begin{lemma}\label{lemma5}
Let $(s_n)_{n\in\mathbb{N}}$ be as defined in Lemma \ref{lemma4}, where $\gamma\in (0, 1/2)$, and $$M(s)=\lfloor \log 1/s/\log\log 1/s\rfloor,\quad s\in (0,\eee^{-\eee}).$$ For $s\in [s_{n+1},s_n]$, the following limit relation holds
\begin{equation*}
\lim_{n\to\infty}L(s)\Big(\sum_{p\ge M(s)+1}\frac{\eta_p}{p^{1/2+s}}-\sum_{p\ge M(s_{n+1})+1}\frac{\eta_p}{p^{1/2+s_{n+1}}}\Big)=0\quad\text{\rm a.s.}
\end{equation*}
\end{lemma}
\begin{proof}
Using the fact that $M$ is a nonincreasing function for the arguments close to $0$, write, for $s\in [s_{n+1},s_n]$,
\begin{multline*}
\sum_{p\ge M(s)+1}\frac{\eta_p}{p^{1/2+s}}-\sum_{p\ge M(s_{n+1})+1}\frac{\eta_p}{p^{1/2+s_{n+1}}}\\=\sum_{M(s)+1\leq p\leq M(s_{n+1})}\frac{\eta_p}{p^{1/2+s}}+\sum_{k\ge M(s_{n+1})+1} \Big(\frac{1}{p^{1/2+s}}-\frac{1}{p^{1/2+s_{n+1}}}\Big)\eta_p=:I_{n,1}(s)+I_{n,2}(s).
\end{multline*}

\noindent {\sc Analysis of $I_{n,1}(s)$.} It will be proved more than we need, namely, $$\lim_{n\to\infty}\sup_{s\in[s_{n+1},\,s_n]}|I_{n,1}(s)|=0\quad \text{a.s.}$$ We obtain with the help of integration by parts
\begin{equation*}
I_{n,1}(s)=\frac{T^\ast(M(s_{n+1}))}{(M(s_{n+1}))^{1/2+s}}-\frac{T^\ast(M(s)+1)}{(M(s)+1)^{1/2+s}}+(1/2+s)\int_{M(s)+1}^{M(s_{n+1})}\frac{T^\ast(x)}{x^{3/2+s}}{\rm d}x+o(1),
\end{equation*}
where, as before, $T^\ast(x)=\sum_{p\leq x}\eta_p$ for $x\geq 1$. The term $o(1)$ is equal to $0$ if $M(s)+1\notin \mathcal{P}$ and is equal to $\eta_{M(s)+1}(M(s)+1)^{-1/2-s}$, otherwise. In view of \eqref{eq:growth} and \eqref{eq:zero}, $\lim_{s\to 0+}\eta_{M(s)+1}(M(s)+1)^{-1/2-s}=0$ a.s. Invoking formula \eqref{5.3} and $\lim_{n\to\infty}(M(s_{n+1}))^{s_{n+1}}=1$ we infer
\begin{multline*}
\frac{|T^\ast(M(s_{n+1}))|}{(M(s_{n+1}))^{1/2+s}}\le \frac{|T^\ast(M(s_{n+1}))|}{(M(s_{n+1}))^{1/2+s_{n+1}}}~\sim~ \frac{|T^\ast(M(s_{n+1}))|}{(M(s_{n+1}))^{1/2}}\\=O((\log\log M(s_{n+1})/\log M(s_{n+1}))^{1/2})\to 0,\quad n\to\infty\quad\text{\rm a.s.}
\end{multline*}
By a similar argument we conclude that
\begin{equation*}
\lim_{n\to\infty}\sup_{s\in[s_{n+1},\,s_n]}\frac{|T^\ast(M(s)+1)|}{(M(s)+1)^{1/2+s}}=0\quad\text{\rm a.s.}
\end{equation*}
Further, for $s\in [s_{n+1}, s_n]$ and large $n$,
\begin{multline*}
\Bigg|\int_{M(s)+1}^{M(s_{n+1})}\frac{T^\ast(x)}{x^{3/2+s}}{\rm d}x\Bigg| \le \int_{M(s)+1}^{M(s_{n+1})}\frac{|T^\ast(x)|}{x^{3/2+s}}{\rm d}x \le
\big(\sup_{y\le M(s_{n+1})}|T^\ast(y)|\big)\int_{M(s)+1}^{M(s_{n+1})}\frac{{\rm d}x}{x^{3/2+s}}\\ \le
\big(\sup_{y\le M(s_{n+1})}|T^\ast(y)|\big)\frac{2}{(M(s))^{1/2+s}} \le \big(\sup_{y\le M(s_{n+1})}|T^\ast(y)|\big)\frac{2}{(M(s_n))^{1/2+s_{n+1}}}\\ =
O((\log\log M(s_{n+1})/\log M(s_{n+1}))^{1/2})~\to~ 0,\quad n\to\infty\quad\text{\rm a.s.}
\end{multline*}
We have used \eqref{5.3}, $\lim_{n\to\infty}(M(s_{n+1})/M(s_n))=1$ and $\lim_{n\to\infty}(M(s_n))^{s_{n+1}}=1$ for the equality.

\noindent {\sc  Analysis of $I_{n,2}(s)$.} We claim that
\begin{equation}\label{lem4.11}
\lim_{n\to\infty}\sup_{s\geq s_{n+1}}\sum_{p\ge M(s_{n+1})+1}\Big(\frac{1}{p^{1/2+s_{n+1}}}-\frac{1}{p^{1/2+s}}\Big)|\eta_p|\1_{\{|\eta_p|>(p/\log p)^{1/2}\log n\}}\\=0\quad \text{\rm{a.s.}}
\end{equation}
and
\begin{equation}\label{lem4.21}
\lim_{n\to\infty}\sup_{s\geq s_{n+1}}\sum_{p\ge M(s_{n+1})+1}\Big(\frac{1}{p^{1/2+s_{n+1}}}-\frac{1}{p^{1/2+s}}\Big)\mathbb{E}\big[|\eta_p|\1_{\{|\eta_p|>(p/\log p)^{1/2}\log n\}}\big]=0.
\end{equation}

As before, let $p_1<p_2<\ldots$ be the sequence of prime numbers arranged in the order of increase. By the prime number theorem, $p_k/\log p_k\sim k$ as $k\to\infty$. We note for later use that this implies that, given $\delta\in (0,1)$,
\begin{equation}\label{eq:ineq}
p_k/\log p_k\geq \delta^2 k
\end{equation}
for large $k$. By the direct part of the Borel-Cantelli lemma, $\lim_{k\to\infty}(p_k/\log p_k)^{-1/2}|\eta_{p_k}|=0$ a.s. and thereupon $\sup_{k\geq 1}\,(p_k/\log p_k)^{-1/2}|\eta_{p_k}|<\infty$ a.s. Hence, relation \eqref{lem4.11} follows from
\begin{multline*}
\sum_{k:~p_k\ge M(s_{n+1})+1}\Big(\frac{1}{p_k^{1/2+s_{n+1}}}-\frac{1}{p_k^{1/2+s}}\Big)|\eta_{p_k}|\1_{\{|\eta_{p_k}|>(p_k/\log p_k)^{1/2}\log n\}}\\\leq \sum_{k:~p_k\ge M(s_{n+1})+1}\frac{|\eta_{p_k}|}{p_k^{1/2}}\1_{\{(p_k/\log p_k)^{-1/2}|\eta_{p_k}|>\log n\}}\quad\text {{\rm a.s.}}
\end{multline*}
and the fact that the summands on the right-hand side are equal to $0$ for large enough $n$. Relation \eqref{lem4.21} follows along the lines of the proof of \eqref{lem4.2}:
\begin{multline*}
\sum_{k:~p_k\ge M(s_{n+1})+1}\Big(\frac{1}{p_k^{1/2+s_{n+1}}}-\frac{1}{p_k^{1/2+s}}\Big)\mathbb{E}\big[|\eta_{p_k}|\1_{\{|\eta_{p_k}|>(p_k/\log p_k)^{1/2}\log n\}}\big]\\ \le \sum_{k:~p_k\ge M(s_{n+1})+1}\frac{1}{p_k^{1/2}}\me\big[|\eta|\1_{\{|\eta|>\delta k^{1/2}\log n\}}\big] \le \sum_{k\ge M(s_{n+1})+1}\frac{1}{k^{1/2}}\me\big[|\eta|\1_{\{k<(\delta \log n)^{-2}\eta^2\}}\big]\\\leq \me \Big[|\eta|\sum_{k\leq \lfloor (\delta \log n)^{-2}\eta^2\rfloor}\frac{1}{k^{1/2}}\Big] \leq 2\delta^{-1}\me[\eta^2](\log n)^{-1}~\to~0,\quad n\to\infty.
\end{multline*} 
We have used \eqref{eq:ineq} for the first inequality.

Put $\widehat{\eta}_p(n):=\eta_p\1_{\{|\eta_p|\leq (p/\log p)^{1/2}\log n\}}-\me \big[\eta_p\1_{\{|\eta_p|\leq (p/\log p)^{1/2}\log n\}}\big]$ for $p\in\mathcal{P}$ and $n\in\mn$. For $n\in\mathbb{N}$ and small positive $u$, put
\begin{equation*}
Y^\ast_n(u):=\sum_{p\ge M(s_{n+1})+1}\frac{\widehat{\eta}_p(n)}{p^{1/2+u}}.
\end{equation*}
In view of \eqref{lem4.11} and \eqref{lem4.21}, we are left with showing that, for each $s\in[s_{n+1}, s_n]$,
\begin{equation*}
\lim_{n\to\infty}L(s_n)(Y^\ast_n(s)-Y^\ast_n(s_{n+1}))=0\quad\text{\rm a.s.}
\end{equation*}
We shall prove an equivalent limit relation: for each $v\in [v_{n+1}, v_n]$,
\begin{equation}\label{5.11}
\lim_{n\to\infty}L(s_n)(Y_n(v)-Y_n(v_{n+1}))=0\quad\text{\rm a.s.},
\end{equation}
where $Y_n(v):=Y_n^\ast(\exp(-1/v))$, $v_n:=1/(\log 1/s_n)=\exp(-n^{1-\gamma})$ for $n\in\mn$ and $v>0$.

For $j\in\mathbb{N}_0$ and $n\in\mathbb{N}$, put
\begin{equation*}
F_j(n):=\{t_{j,\,m}(n):=v_{n+1}+2^{-j}m(v_n-v_{n+1}):0\le m\le 2^j\}.
\end{equation*}
Note that $F_j(n)\subseteq F_{j+1}(n)$ and put $F(n):=\bigcup_{j\ge 0}F_j(n)$. The set $F(n)$ is dense in the interval $[v_{n+1},v_n]$. For any $u\in[v_{n+1},v_n]$, put
\begin{equation*}
u_j:=\max\{v\in F_j(n): v\le u\}=v_{n+1}+2^{-j}(v_n-v_{n+1})\Big\lfloor\frac{2^j(u-v_{n+1})}{v_n-v_{n+1}}\Big\rfloor.
\end{equation*}
Then $\lim_{j\to\infty} u_j=u$ (we omit the dependence on $n$ in the notation). An important observation is that either $u_{j-1}=u_j$ or $u_{j-1}=u_j-2^{-j}(v_n-v_{n+1})$. Consequently, $u_j=t_{j,m}$ for some $0\le m \le 2^j$, which implies that either $u_{j-1}=t_{j,\,m}$ or $u_{j-1}=t_{j,\,m-1}$. Since $Y_n$ is a.s. continuous on $[v_{n+1}, v_n]$, we obtain
    \begin{multline*}
        |Y_n(u)-Y_n(v_{n+1})|=\lim_{l\to\infty}|Y_n(u_l)-Y_n(v_{n+1})|\\ =\lim_{l\to\infty}\Big|\sum_{j=1}^{l}(Y_n(u_j)-Y_n(u_{j-1}))+Y_n(u_0)-Y_n(v_{n+1})\Big|\\ \le \lim_{l\to\infty}\sum_{j=0}^l\max_{1\le m \le 2^j}|Y_n(t_{j,\,m})-Y_n(t_{j,\,m-1})|\\=\sum_{j\ge 0}\max_{1\le m \le 2^j}|Y_n(t_{j,\,m})-Y_n(t_{j,\,m-1})|.
    \end{multline*}
Thus, our purpose is to prove that, for all $\varepsilon>0$ and sufficiently large $n_0\in\mathbb{N}$,
\begin{equation*}
\sum_{n\ge n_0}\mathbb{P}\Big\{\sum_{j\ge 0}\max_{1\le m \le 2^j}L(s_n)|Y_n(t_{j,\,m})-Y_n(t_{j,\,m-1})|>\varepsilon\Big\}<\infty.
\end{equation*}
Put $a_j:=(j+1)2^{-j/2}$ for $j\in\mathbb{N}_0$. In view of $\sum_{j\ge 0}a_j<\infty$, it suffices to show that, for all $\varepsilon>0$,
\begin{equation}\label{5.12}
\sum_{n\ge n_0}\sum_{j\ge 0}\mathbb{P}\Big\{\max_{1\le m \le 2^j}L(s_n)|Y_n(t_{j,\,m})-Y_n(t_{j,\,m-1})|>\varepsilon a_j \Big\}<\infty.
\end{equation}

The subsequent argument is similar to that used in the proof of Lemma \ref{lemma4}. In view of this, we only give a sketch.  
Write, for $u\in\mathbb{R}$ and sufficiently large $n$,
\begin{multline*}
\me\big[\exp\big(\pm u(Y_n(t_{j,\,m})-Y_n(t_{j,\,m-1}))\big)\big]\\=\mathbb{E}\Big[\exp\Big(\pm u \sum_{p\ge M(s_{n+1})+1}\frac{1}{p^{1/2}}\Big(\frac{1}{p^{\exp(-1/t_{j,\,m})}}-\frac{1}{p^{\exp(-1/t_{j,\,m-1})}}\Big)\widehat{\eta}_p(n) \Big)\Big] \\ \le \prod_{p\ge M(s_{n+1})+1}\Big(1+\frac{u^2}{2p} 
\Big(\frac{1}{p^{\exp(-1/t_{j,\,m})}}-\frac{1}{p^{\exp(-1/t_{j,\,m-1})}}\Big)^2\\\times\mathbb{E}\Big[(\widehat{\eta}_p(n))^2\exp\Big(\frac{|u|}{p^{1/2} 
}\Big(\frac{1}{p^{\exp(-1/t_{j,\,m})}}-\frac{1}{p^{\exp(-1/t_{j,\,m-1})}}\Big)|\widehat{\eta}_p(n)|\Big)\Big].
\end{multline*}
Now we prove that, for large $n$, all $j\in\mn_0$ and integer $m\in [0,2^j]$ and some constant $C>0$,
\begin{equation*}
A_j(n):=\sum_{p\ge M(s_{n+1})+1}\frac{1}{p}\Big(\frac{1}{p^{\exp(-1/t_{j,\,m})}}-\frac{1}{p^{\exp(-1/t_{j,\,m-1})}}\Big)^2 \leq \frac{C2^{-j}(v_n-v_{n+1})}{v_{n+1}^2}.
\end{equation*}
For $x>\eee$, put $L_{j,\,m}(x)=(x\log x)^{-1}(x^{-\exp(-1/t_{j,\,m})}-x^{-\exp(-1/t_{j,\,m-1})})^2$. The function $L_{j,\,m}$ is differentiable and decreasing on $(\eee,\infty)$, whence $L^\prime_{j,\,m}(x)\leq 0$ for $x>\eee$. Also, for $x\geq \eee$, put $A(x):=\int_{(\eee,\,x]}\log x\, {\rm d}\rho(x)$. By the prime number theorem, for $x_0$ large enough (in particular, $x_0>\eee$) there exists $C_0>0$ such that $A(x)\leq C_0x$ for $x\geq x_0$. Hence, there exists $C_1>C_0$ such that $A(x)\leq C_1(x-\eee)$ for $x\geq x_0$. Also, there exists $C_2>0$ such that $A(x)\leq C_2(x-e)$ for $x\in [\eee,\,x_0]$. Thus, putting $C:=\max (C_1, C_2)$ we infer $A(x)\leq C(x-\eee)$ for $x\geq \eee$. Integrating by parts twice, forward and then backward, we obtain
\begin{multline*}
A_j(n) \leq \int_\eee^\infty L_{j,\,m}(x){\rm d}A(x)=\int_\eee^\infty (-L^\prime_{j,\,m}(x))A(x){\rm d}x\leq C\int_\eee^\infty  (-L^\prime_{j,\,m}(x))(x-\eee){\rm d}x\\=C\int_\eee^\infty L_{j,\,m}(x){\rm d}x.
\end{multline*}
According to formula (17) in \cite{Iksanov+Kostohryz:2025}, $$\int_\eee^\infty L_{j,\,m}(x){\rm d}x\leq \frac{1}{t_{j,\,m-1}}-\frac{1}{t_{j,\,m}}\leq \frac{2^{-j}(v_n-v_{n+1})}{v_{n+1}^2}.$$ This proves the claimed inequality for $A_j(n)$.

Next, we estimate $$B_j(u,p,n):=\frac{|u|}{p^{1/2}}\Big(\frac{1}{p^{\exp(-1/t_{j,\,m-1})}}-\frac{1}{p^{\exp(-1/t_{j,\,m})}}\Big)|\widehat{\eta}_p(n)|$$ for $p\geq \lfloor M(s_{n+1})\rfloor+1$. Assume first that $2^j\geq v_{n+1}^{-2}(v_n-v_{n+1})$ for integer $j$. For $a>0$ and $0<s<t$, the following 
holds $$0\leq \exp(-a\eee^{-1/s})-\exp(-a\eee^{-1/t})\leq a\exp(-a\eee^{-1/s})\eee^{-1/t}(1/s-1/t).$$ Putting
$a=\log p$, $t=t_{j,\,m-1}$ and $s=t_{j,\,m}$ we obtain $$0\leq \frac{1}{p^{\exp(-1/t_{j,\,m-1})}}-\frac{1}{p^{\exp(-1/t_{j,\,m})}}\leq \frac{\log p}{p^{\exp(-1/t_{j,\,m-1})}}\eee^{-1/t_{j,\,m}}\Big(\frac{1}{t_{j,\,m-1}}-\frac{1}{t_{j,\,m}}\Big).$$ Using $x\eee^{-ax}\leq (\eee a)^{-1}$ for $x\geq 0$ and our present assumption concerning $j$ we further conclude that the right-hand side does not exceed
\begin{multline*}
\frac{1}{\eee}\eee^{1/t_{j,\,m-1}-1/t_{j,\,m}}\Big(\frac{1}{t_{j,\,m-1}}-\frac{1}{t_{j,\,m}}\Big)\\ \leq \frac{1}{\eee} \exp\Big(\frac{2^{-j}(v_n-v_{n+1})}{v_{n+1}^2}\Big)\frac{2^{-j}(v_n-v_{n+1})}{v_{n+1}^2}\leq \frac{2^{-j}(v_n-v_{n+1})}{v_{n+1}^2}.
\end{multline*}
We have shown that $$B_j(u,p,n)\leq \frac{|u|2^{1-j}\log n}{(\log M(s_{n+1}))^{1/2}} \frac{v_n-v_{n+1}}{v_{n+1}^2}=:C_j(u,n)$$ whenever $p\geq \lfloor M(s_{n+1})\rfloor+1$ and $2^j\geq v_{n+1}^{-2}(v_n-v_{n+1})$.

Assume now that $2^j\leq v_{n+1}^{-2}(v_n-v_{n+1})$ for nonnegative integer $j$. In view of $$\frac{1}{p^{\exp(-1/t_{j,\,m-1})}}-\frac{1}{p^{\exp(-1/t_{j,\,m})}}\leq 1$$ we infer
$$B_j(u,p,n)\leq \frac{2|u|\log n}{(\log M(s_{n+1}))^{1/2}}=:C_j(u,n)$$ whenever $p\geq \lfloor M(s_{n+1})\rfloor+1$.

Using now $1+x\le \eee^x$ for $x\in\mathbb{R}$ we arrive at
\begin{equation*}
\me\big[\exp\big(\pm u(Y_n(t_{j,\,m})-Y_n(t_{j,\,m-1}))\big)\big] \leq \exp\Big(\frac{u^2}{2}\frac{C2^{-j}(v_n-v_{n+1})}{v_{n+1}^2}\eee^{C_j(u,n)}\Big),\quad u\in\mathbb{R}.
\end{equation*}
By Markov's inequality and the inequality $\eee^{u|x|}\le \eee^{ux}+\eee^{-ux}$ for $x\in\mathbb{R}$,
\begin{multline*}
\mathbb{P}\big\{L(s_n)|Y_n(t_{j,\,m})-Y_n(t_{j,\,m-1})|>\varepsilon a_j \big\}\leq \exp\Big(-u\frac{\varepsilon a_j}{L(s_n)}\Big)\mathbb{E}[\exp(u|Y_n(t_{j,\,m})-Y_n(t_{j,\,m-1})|)]\\\leq  2\exp\Big(-\frac{u\varepsilon a_j}{L(s_n)}+\frac{u^2}{2}\frac{C2^{-j}(v_n-v_{n+1})}{v_{n+1}^2}\eee^{C_j(u,n)}\Big).
\end{multline*}
Put $$u=\frac{\varepsilon 2^{j/2}}{CL(s_n)}\frac{v_{n+1}^2}{v_n-v_{n+1}}\quad\text{and}\quad k_n:=\frac{1}{C(L(s_n))^2}\frac{v_{n+1}^2}{v_n-v_{n+1}}.$$ As a consequence of $$(\log M(s_{n+1}))^{-1/2}~\sim~ n^{\gamma/2-1/2},~\Big(\frac{v_{n+1}}{v_n-v_{n+1}}\Big)^{1/2}~\sim~(1-\gamma)^{-1/2} n^{\gamma/2},\quad n\to\infty$$ and $$\lim_{n\to\infty} (\log 1/s_n)v_{n+1}=1,$$ we infer, for $j$ satisfying $2^{j}\geq v_{n+1}^{-2}(v_n-v_{n+1})$,
\begin{multline*}
C_j(u,n)=\frac{|u|2^{1-j}\log n}{(\log M(s_{n+1}))^{1/2}}\frac{v_n-v_{n+1}}{v_{n+1}^2}=\frac{\varepsilon 2^{1-j/2}\log n}{CL(s_n)(\log M(s_{n+1}))^{1/2}}\\\leq \frac{2\varepsilon \log n}{CL(s_n)(\log M(s_{n+1}))^{1/2}}\frac{v_{n+1}}{(v_n-v_{n+1})^{1/2}}\\=\frac{\varepsilon 2^{3/2}((\log 1/s_n) v_{n+1})^{1/2}(\log^{(3)}1/s_n)^{1/2}\log n}{C(\log M(s_{n+1}))^{1/2}}\Big(\frac{v_{n+1}}{v_n-v_{n+1}}\Big)^{1/2}\\~\sim~ \varepsilon C^{-1} 2^{3/2}n^{\gamma-1/2}(\log n)^{3/2}~\to~0,\quad n\to\infty.
\end{multline*}
If nonnegative integer $j$ satisfies $2^{j}\leq v_{n+1}^{-2}(v_n-v_{n+1})$, then
\begin{multline*}
C_j(u,n)=\frac{2|u|\log n}{(\log M(s_{n+1}))^{1/2}}=\frac{\varepsilon 2^{1+j/2}\log n}{CL(s_n)(\log M(s_{n+1}))^{1/2}}\frac{v_{n+1}^2}{v_n-v_{n+1}}\\\leq \frac{2\varepsilon \log n}{CL(s_n)(\log M(s_{n+1}))^{1/2}}\frac{v_{n+1}}{(v_n-v_{n+1})^{1/2}}~\to~0,\quad n\to\infty.
\end{multline*}
Thus, we obtain, for large $n$ satisfying $k_n>\varepsilon^{-2} \log 2$ and $\eee^{C_j(u,n)}\leq 3/2$,
\begin{multline*}
\sum_{j\ge 0}\mathbb{P}\Big\{\max_{1\le m \le 2^j}L(s_n)|Y_n(t_{j,\,m})-Y_n(t_{j,\,m-1})|>\varepsilon a_j \Big\}\\ \le \sum_{j\ge 0}2^j 2\exp(-\varepsilon^2(j+1)k_n)\exp\big(3\varepsilon^2k_n/4\big)=\frac{2\exp(-\varepsilon^2 k_n/4)}{1-2\exp(-\varepsilon^2 k_n)}.
\end{multline*}
Since $k_n\sim 2C^{-1} n^\gamma\log n$ as $n\to\infty$, \eqref{5.12} follows.
\end{proof}

After all these preparations we are ready to prove Proposition \ref{pr1}.
\begin{proof}[Proof of Proposition \ref{pr1}]

Recall our convention that $\sigma^2=1$. To prove \eqref{5.1}, we choose any $\gamma>0$ sufficiently close to $0$, put $s_n=\exp(-\exp(n^{1-\gamma}))$ for $n\in\mn$ and select $\theta=\theta(\gamma)$ such that \eqref{5.7} holds true. Let $M(s)=\lfloor \log 1/s/\log\log 1/s\rfloor$ for $s\in (0,\eee^{-\eee})$.
Using Lemmas \ref{lemma_3} and \ref{lemma4} in combination with $\mathbb{E}\big[\eta_p\1_{(A_{p,\theta}(s))^c}\big]=-\me\big[\eta_p\1_{A_{p,\theta}(s)}\big]$ we conclude that
\begin{equation}\label{5.16}
\limsup_{n\to\infty} L(s_n)\sum_{p\ge M(s_n)+1}\frac{\eta_p}{p^{1/2+s_n}}\le 1+\gamma\quad\text{\rm a.s.}
\end{equation}
Relation \eqref{5.16} together with Lemma \ref{lemma5} ensures that
$$\limsup_{s\to 0+} L(s)\sum_{p\ge M(s)+1}\frac{\eta_p}{p^{1/2+s}}\le 1+\gamma\quad\text{\rm a.s.}$$ With these at hand, an application of Lemma \ref{lemma_2} yields
\begin{equation*}
\limsup_{s\to 0+}L(s)\sum_p \frac{\eta_p}{p^{1/2+s}}\le 1+\gamma\quad\text{\rm a.s.}
\end{equation*}
Sending $\gamma\to 0+$ we arrive at \eqref{5.1}.

Relation \eqref{5.2} follows from \eqref{5.1}, with $-\eta_p$ replacing $\eta_p$.
\end{proof}

\begin{proposition}\label{pr2}
Assume that $\me[\eta]=0$ and $\sigma^2=\me[\eta^2]\in(0,\infty)$. Then
\begin{equation}\label{5.17}
\limsup_{s\to 0+}\Big(\frac{1}{\log 1/s\,\log^{(3)}1/s}\Big)^{1/2}\sum_p \frac{\eta_p}{p^{1/2+s}}\geq \sqrt{2}\sigma 
\quad\text{\rm a.s.}
\end{equation}
and
\begin{equation}\label{5.18}
\liminf_{s\to 0+}\Big(\frac{1}{\log 1/s \,\log^{(3)}1/s}\Big)^{1/2}\sum_p \frac{\eta_p}{p^{1/2+s}}\leq-\sqrt{2}\sigma 
\quad\text{\rm a.s.}
\end{equation}
\end{proposition}

We briefly explain the strategy behind the proof of 
\eqref{5.17}. The sum $\sum_p \frac{\eta_p}{p^{1/2+s}}$ will now be split into three fragments: initial, intermediate and final. It will be shown in Lemma \ref{lemma5.7} that the initial and the final fragments do not contribute to the LIL provided that the variables $\eta_p$ within each fragment are properly truncated and centered. Lemma \ref{lemma5.8} treats the intermediate fragment which gives a principal contribution to the LIL. Our proof is based on the converse part of the Borel-Cantelli lemma, which requires independence. The independence requirement complicates to some extent a selection of the intermediate fragment. Unlike in the proof of \eqref{5.1}, it suffices to prove \eqref{5.17} for a suitable sequence. This is a simplifying feature of \eqref{5.17} in comparison to \eqref{5.1}.

In what follows we assume without further notice that $\me [\eta]=0$ and $\sigma^2={\rm Var}\,[\eta]=1$. Also, we use the sets $A_{p,\theta}(s)$ and the corresponding variables $\widetilde{\eta}_{p,\theta}(s)$ with $\theta=1$.
\begin{lemma}\label{lemma5.7}
Fix any $\gamma>0$ and put $\mathfrak{s}_n:=\exp(-\exp(n^{1+\gamma}))$ for $n\ge 1$. Let $N_1$ and $N_2$ be functions which take positive integer values, may depend on $\gamma$ and satisfy\footnote{For instance, one can take $N_1(s)=\lfloor (\log 1/s)^{1/2}\rfloor$ or $N_1(s)=\lfloor\exp(\eee^{\log 1/s/\log\log 1/s})\rfloor$ and $N_2(s)=\lfloor \exp((1/s)\log 1/s)\rfloor$.} $\lim_{s\to 0+}N_1(s)=\infty$, $\lim_{s\to 0+}(\log\log N_1(s)/\log 1/s)=0$, $\lim_{s\to 0+}(\log\log N_2(s)/\log 1/s)=1$ and $\lim_{s\to 0+}s\log N_2(s)=\infty$. Then
\begin{equation}\label{5.19}
\lim_{n\to\infty}L(\mathfrak{s}_n) \sum_{p\leq N_1(\mathfrak{s}_n)}\frac{\widetilde{\eta}_{p,1}(\mathfrak{s}_n)}{p^{1/2+\mathfrak{s}_n}}=0\quad\text{\rm a.s.}
\end{equation}
and
\begin{equation}\label{5.20}
\lim_{n\to\infty}L(\mathfrak{s}_n) \sum_{p\ge N_2(\mathfrak{s}_n)+1}\frac{\widetilde{\eta}_{p,1}(\mathfrak{s}_n)}{p^{1/2+\mathfrak{s}_n}}=0\quad\text{\rm a.s.}
\end{equation}
\end{lemma}
\begin{proof}
As in the proof of Lemma \ref{lemma4}, we obtain the result, with $L^\ast$ replacing $L$. For $s>0$ close to $0$, put
\begin{equation*}
Z_1(s):=L^\ast(s)\sum_{p\leq N_1(s)}\frac{\widetilde{\eta}_{p,1}(s)}{p^{1/2+s}}.
\end{equation*}
The reasoning used to derive both \eqref{5.19} and \eqref{5.20} is similar to the one applied in the proof of Lemma \ref{lemma4}. Therefore, we give a proof of \eqref{5.19} and indicate the only minor change needed for a proof of
\eqref{5.20}.

Regarding \eqref{5.19}, in view of the direct part of the Borel–Cantelli lemma, it is sufficient to demonstrate that, for all $\varepsilon>0$,
\begin{equation}\label{5.21}
\sum_{n\ge 1}\mathbb{P}\{Z_1(\mathfrak{s}_n)>\varepsilon\}<\infty.
\end{equation}
To this end, we obtain a counterpart of \eqref{5.9}, for $u\in\mathbb{R}$,
\begin{equation*}
\mathbb{E}[\eee^{uZ_1(s)}]\le\exp\Big(\frac{u^2(L^\ast(s))^2}{2}\sum_{p\leq N_1(s)}\frac{1}{p^{1+2s}}\exp\Big(\frac{\sqrt{2}|u|}{ \log^{(3)}1/s}\Big)\Big).
\end{equation*}
Recalling the notation $\rho(x)=\#\{p\in\mathcal{P}: p\leq x\}$ for $x\geq 1$ we obtain with the help of integration by parts
\begin{equation*}
\sum_{p\leq N_1(s)}p^{-1-2s}=\int_{(1,\,N_1(s)]}x^{-1-2s}{\rm d}\rho(x)= (N_1(s))^{-1-2s}\rho(N_1(s))+(1+2s)\int_1^{N_1(s)}x^{-2-2s} \rho(x){\rm d}x.
\end{equation*}
The assumption $\lim_{s\to 0+}(\log\log N_1(s)/\log 1/s)=0$ entails $\lim_{s\to 0+}s\log N_1(s)=0$. By the prime number theorem, $\rho(x)\sim x/\log x$ as $x\to\infty$. As a consequence, $$\lim_{s\to 0+}(N_1(s))^{-1-2s}\rho(N_1(s))=0.$$ By another application of the prime number theorem, as $s\to 0+$, $$\int_1^{N_1(s)}x^{-2-2s} \rho(x){\rm d}x~\sim~\int_2^{N_1(s)}x^{-1-2s}(\log x)^{-1}{\rm d}x=\int_{(2\log 2) s}^{2s\log N_1(s)} y^{-1}\eee^{-y}{\rm d}y~\sim~ \log\log N_1(s).$$ Thus, we have proved that
\begin{equation}\label{5.22}
\sum_{p\leq N_1(s)}p^{-1-2s}~\sim~ \log\log N_1(s),\quad s\to 0+.
\end{equation}
By assumption, $\lim_{s\to 0+}(\log\log N_1(s)/g(s))=0$. Hence, there exists an $r>0$ close to $0$ 
such that
\begin{equation*}
\sum_{p\leq N_1(s)}p^{-1-2s}\le r g(s)
\end{equation*}
for small positive $s$, and also $(1-\delta)(1+\gamma)>1$, where $\delta:=r(4\varepsilon^2)^{-1}\exp(\sqrt{2}\varepsilon^{-1})$ with $\varepsilon$ appearing in \eqref{5.21}. Then,
for $u\in\mathbb{R}$ and small $s>0$,
\begin{multline*}
\mathbb{E}[\eee^{uZ_1(s)}]\le\exp\Big(\frac{ru^2 g(s)(L^\ast(s))^2}{2}\exp\Big(\frac{\sqrt{2}|u|}{ \log^{(3)}1/s}\Big)\Big)=\exp\Big(\frac{ru^2}{4 \log^{(3)}1/s}\exp\Big(\frac{\sqrt{2}|u|}{ \log^{(3)}1/s}\Big)\Big).
\end{multline*}
Put $u=(1/\varepsilon)\log^{(3)} 1/\mathfrak{s}_n$. By Markov's inequality, we infer, for large $n$,
\begin{equation*}
\mathbb{P}\{Z_1(\mathfrak{s}_n)>\varepsilon\}\le \eee^{-u\varepsilon}\mathbb{E}[\eee^{uZ_1(\mathfrak{s}_n)}]\le\exp\big(-\big(1-r(4\varepsilon^2)^{-1}\eee^{\sqrt{2}\varepsilon^{-1}}\big)\log^{(3)} 1/\mathfrak{s}_n\big)=\frac{1}{n^{(1-\delta)(1+\gamma)}}.
\end{equation*}
This proves \eqref{5.21} and \eqref{5.19}.

We omit a proof of \eqref{5.20}, which is analogous to that of \eqref{5.19}, and only point out a counterpart of \eqref{5.22}:
\begin{equation}\label{5.23}
\sum_{p\ge N_2(s)+1} p^{-1-2s}=o(1),\quad s\to 0+.
\end{equation}
This is justified as follows:
$$\sum_{p\ge N_2(s)+1} p^{-1-2s}=\int_{N_2(s)+1}^\infty x^{-1-2s}{\rm d}\rho(x)=-\frac{\rho(N_2(s)+1)}{(N_2(s)+1)^{1+2s}}+(1+2s)\int_{N_2(s)+1}^\infty x^{-2-2s}\rho(x){\rm d}x.$$ The first summand is $o(1)$ as $s\to 0+$. By the prime number theorem, as $s\to 0+$, $$\int_{N_2(s)+1}^\infty x^{-2-2s}\rho(x){\rm d}x~\sim~ \int_{N_2(s)+1}^\infty x^{-1-2s}(\log x)^{-1}{\rm d}x=\int_{2s\log (N_2(s)+1)}^\infty \eee^{-y}y^{-1}{\rm d}y=o(1).$$ We have used the assumption $\lim_{s\to 0+} s\log (N_2(s)+1)=\infty$ for the last equality.
\end{proof}

\begin{lemma}\label{lemma5.8}
Fix sufficiently small $\delta>0$, pick $\gamma>0$ satisfying \newline $(1+\gamma)(1-\delta^2/8)<1$ and let, as before, $\mathfrak{s}_n=\exp(-\exp(n^{1+\gamma}))$ for $n\in\mn$. Then
\begin{equation*}
\limsup_{n\to\infty} L(\mathfrak{s}_n) \sum_p \frac{\widetilde{\eta}_{p,1}(\mathfrak{s}_n)}{p^{1/2+\mathfrak{s}_n}}\ge 1-\delta\quad\text{\rm a.s.}
\end{equation*}
\end{lemma}
\begin{proof}
For any function $r$ such that $\lim_{s\to 0+}r(s)=+\infty$, we can choose a function $N_1$ in Lemma \ref{lemma5.7} satisfying $$\lim_{s\to 0+}\frac{\log\log N_1(s)}{\log 1/s}r(s)=+\infty.$$ Thus, relation $\lim_{n\to\infty}(\mathfrak{s}_{n+1}/\mathfrak{s}_n)=\infty$ ensures that there exists an $N_1$ satisfying $$\frac{\log\log N_1(\mathfrak{s}_{n+1})}{\log 1/\mathfrak{s}_n}=\frac{\log\log N_1(\mathfrak{s}_{n+1})}{\log 1/\mathfrak{s}_{n+1}}\frac{\log 1/\mathfrak{s}_{n+1}}{\log 1/\mathfrak{s}_n}~\to~+\infty,\quad n\to\infty.$$ For instance, one can take $N_1(s)=\lfloor\exp(\eee^{\log 1/s/\log\log 1/s})\rfloor$.

Let $N_2$ be any function satisfying the assumptions of Lemma \ref{lemma5.7}. Since $$\lim_{n\to\infty}\frac{\log\log N_2(\mathfrak{s}_n)}{\log 1/\mathfrak{s}_n}=1,$$ we conclude that there exists $n_0\in\mn$ such that $$\frac{\log\log N_1(\mathfrak{s}_{n+1})}{\log 1/\mathfrak{s}_n}\geq \frac{\log\log N_2(\mathfrak{s}_n)}{\log 1/\mathfrak{s}_n}\quad\text{for all}~n\geq n_0,$$ whence
\begin{equation}\label{5.24}
N_1(\mathfrak{s}_{n+1})\ge N_2(\mathfrak{s}_n)\quad\text{for all}~n\ge n_0.
\end{equation}

Put, for small $s>0$,
\begin{equation*}
Z_2(s):=L(s)\sum_{N_1(s)<p\leq N_2(s)}\frac{\widetilde{\eta}_{p,1}(s)}{p^{1/2+s}}.
\end{equation*}
Lemma \ref{lemma5.7} ensures it is enough to prove that
\begin{equation}\label{5.25}
\limsup_{n\to\infty}Z_2(\mathfrak{s}_n)\ge 1-\delta\quad\text{a.s.}
\end{equation}
Our plan is to show that there exists $\overline{s}>0$ such that, for all $s\in(0,\overline{s})$,
\begin{equation}\label{5.26}
\mathbb{P}\{Z_2(s)>1-\delta\}\ge 3^{-1}\eee^{-(1-\delta^2/8)\log^{(3)}1/s}.
\end{equation}
Pick $n_1\ge n_0$ such that $\mathfrak{s}_n<\overline{s}$ for $n\ge n_1$. Inequality \eqref{5.26} together with $(1+\gamma)(1-\delta^2/8)<1$ secures
\begin{equation*}
\sum_{n\ge n_1}\mathbb{P}\{Z_2(\mathfrak{s}_n)>1-\delta\}\ge 3^{-1}\sum_{n\ge n_1}\frac{1}{n^{(1+\gamma)(1-\delta^2/8)}}=\infty.
\end{equation*}
In view of \eqref{5.24}, the random variables $Z_2(\mathfrak{s}_{n_1})$, $Z_2(\mathfrak{s}_{n_1+1}),\ldots$ are independent. Thus, by the converse part of the Borel-Cantelli lemma, divergence of the latter series entails \eqref{5.25}.

When proving \eqref{5.26} we use the event
\begin{equation*}
U_s:=\big\{1-\delta<Z_2(s)\le 1\big\}=\big\{(1-\delta)V(s)< W(s)/(g(s))^{1/2}\le V(s)\big\},
\end{equation*}
where $V(s)=(2\,\log^{(3)}1/s)^{1/2}$ and
\begin{equation*}
W(s):=\frac{Z_2(s)}{L(s)}=\sum_{N_1(s)< p\leq N_2(s)}\frac{\widetilde{\eta}_{p,1}(s)}{p^{1/2+s}}.
\end{equation*}
For $u\in\mathbb{R}$ and small $s>0$, let $\mathbb{Q}_{s,u}$ be a probability measure on $(\Omega, \mathfrak{F})$ defined by
\begin{equation}
\mathbb{Q}_{s,u}(A)=\frac{\mathbb{E}\big[\eee^{uW(s)/(g(s))^{1/2}}\1_A\big]}{\mathbb{E}[\eee^{uW(s)/(g(s))^{1/2}}]}, \quad A\in\mathfrak{F}.
\end{equation}
Then
\begin{multline}\label{5.27}
\mathbb{E}\big[\eee^{u(W(s)/(g(s))^{1/2}-V(s))}\big]\mathbb{Q}_{s,u}(U_s)=\eee^{-uV(s)}\me\big[\eee^{uW(s)/(g(s))^{1/2}}\1_{U_s}\big]\\\leq \mathbb{P}(U_s)\le\mathbb{P}\{Z_2(s)>1-\delta\}.
\end{multline}
We claim that
\begin{equation}\label{eq:second}
\mathbb{Q}_{s,u}(U_s)\ge 1/3
\end{equation}
and
\begin{equation}\label{eq:first}
\mathbb{E}\big[\eee^{u(W(s)/(g(s))^{1/2}-V(s))}\big]\geq \eee^{-(1-\delta^2/8)\log^{(3)}1/s}
\end{equation}
provided that
\begin{equation}\label{eq:choiceu}
u=u(s)=\sqrt{2}(1-\delta/2)(\log^{(3)}1/s)^{1/2}.
\end{equation}
The argument given in the proof of Lemma 6 in \cite{Iksanov+Kostohryz:2025}, which justifies an inequality analogous to \eqref{eq:second}, applies without any changes. Hence, it is omitted. Now we show that \eqref{eq:first} holds. Inequality \eqref{5.25} is an immediate consequence of \eqref{eq:second} and \eqref{eq:first}.

\noindent {\sc Proof of \eqref{eq:first}.} We first prove that, with $u=O((\log^{(3)}1/s)^{1/2})$,
\begin{equation}\label{5.28}
\mathbb{E}\big[\eee^{uW(s)/(g(s))^{1/2}}\big]~\sim~ \eee^{u^2/2+u^2 h(s)},\quad s\to 0+
\end{equation}
for some function $h$ satisfying $\lim_{s\to 0+}h(s)=0$. Put
\begin{equation*}
\xi_p(s)=\frac{\widetilde{\eta}_{p,1}(s)}{(g(s))^{1/2}p^{1/2+s}},\quad p\in\mathcal{P}, \quad s\in(0, \eee^{-\eee}).
\end{equation*}
As a consequence, for $u\in\mathbb{R}$,
\begin{equation*} 
uW(s)/(g(s))^{1/2}=\sum_{N_1(s)<p\leq N_2(s)}u\xi_p(s).
\end{equation*}
According to the second inequality in \eqref{5.8},
\begin{equation*}
|u\xi_p(s)|=O(1/\log\log 1/s)=o(1),\quad s\to 0+\quad\text{\rm a.s.}
\end{equation*}
for each $p\in\mathcal{P}$. Invoking
\begin{equation*}
\eee^x=1+x+x^2/2+o(x^2)\quad\text{and}\quad\log(1+x)=x+O(x^2),\quad x\to 0,
\end{equation*}
we obtain
\begin{multline}\label{5.30}
\mathbb{E}\big[\eee^{uW(s)/(g(s))^{1/2}}\big]=\prod_{N_1(s)< p\leq N_2(s)}\mathbb{E}[\exp(u\xi_p(s))]\\=\prod_{N_1(s)<p\leq N_2(s)}\mathbb{E}\big[1+u\xi_p(s)+u^2\xi_p^2(s)(1/2+o(1))\big]\\=
\exp\sum_{N_1(s)<p\leq N_2(s)}\log\big(1+u^2\mathbb{E}\big[\xi_p^2(s)(1/2+o(1))\big]\big)\\=
\exp\Big(u^2\big(1/2+o(1)\big)\sum_{N_1(s)<p\leq N_2(s)}\mathbb{E}[\xi_p^2(s)]+u^4O\Big(\sum_{N_1(s)<p\leq N_2(s)}(\mathbb{E}[\xi_p^2(s)])^2\Big)\Big).
\end{multline}
Here, we have used the fact that the $o(1)$ term can be chosen nonrandom. In view of \eqref{5.22} and \eqref{5.23},
\begin{equation*}
\sum_{N_1(s)<p\leq N_2(s)}p^{-1-2s}~\sim~ g(s),\quad s\to 0+.
\end{equation*}
The limit relation
\begin{equation*}
\mathbb{E}[\widetilde{\eta}_{p,1}^2(s)]=\mathbb{E}[\eta_p^2\1_{(A_{p,1}(s))^c}]-\big(\mathbb{E}[\eta_p\1_{(A_{p,1}(s))^c}]\big)^2~\to~ 1,\quad s\to 0+
\end{equation*}
holds uniformly in $p\in[N_1(s)+1,N_2(s)]$, $p\in\mathcal{P}$. A combination of these two facts yields
\begin{equation}\label{5.31}
\sum_{N_1(s)< p\leq N_2(s)}\mathbb{E}[\xi_p^2(s)]=\frac{1}{g(s)}\sum_{N_1(s)<p\leq N_2(s)}\frac{\mathbb{E}[\widetilde{\eta}_{p,1}^2(s)]}{p^{1+2s}}~\to~ 1,\quad s\to 0+.
\end{equation}
Finally,
\begin{multline}\label{5.32}
u^2\sum_{N_1(s)<p\leq N_2(s)}(\mathbb{E}[\xi_p^2(s)])^2=\frac{u^2}{(g(s))^2}\sum_{N_1(s)<p\leq N_2(s)}\frac{(\mathbb{E}[\widetilde{\eta}_{p,1}^2(s)])^2}{p^{2+4s}}\\
\le \frac{u^2}{(g(s))^2}\sum_{p>N_1(s)+1}\frac{1}{p^2}=o(1),\quad s\to 0+.
\end{multline}
Now \eqref{5.28} follows from \eqref{5.30}, \eqref{5.31} and \eqref{5.32}.

Finally, formula \eqref{5.28}, with $u$ as in \eqref{eq:choiceu}, entails
\begin{equation*}
\mathbb{E}[\eee^{u(W(s)/(g(s))^{1/2}-V(s))}]=\eee^{-(1-\delta^2/4)\log^{(3)}1/s+o(\log^{(3)}1/s)}\ge\eee^{-(1-\delta^2/8)\log^{(3)}1/s}
\end{equation*}
for small $s>0$.

The proof of Lemma \ref{lemma5.8} is complete.
\end{proof}
\begin{proof}[Proof of Proposition \ref{pr2}]
To prove \eqref{5.17}, pick sufficiently small $\delta>0$ and $\gamma>0$, and put $\mathfrak{s}_n=\exp(-\exp(n^{1+\gamma}))$ for $n\geq 1$. Using $\mathbb{E}[\eta]=0$, write 
\begin{multline*}
L(\mathfrak{s}_n) \sum_p\frac{\eta_p}{p^{1/2+\mathfrak{s}_n}}=L(\mathfrak{s}_n) \sum_{p\leq \lfloor(\log 1/\mathfrak{s}_n)^{1/2}\rfloor}\frac{\eta_p}{p^{1/2+\mathfrak{s}_n}}\\-L(\mathfrak{s}_n) \sum_{p\leq \lfloor(\log 1/\mathfrak{s}_n)^{1/2}\rfloor}\frac{1}{p^{1/2+\mathfrak{s}_n}}(\eta_p\1_{(A_{p,1}(\mathfrak{s}_n))^c}-\mathbb{E}[\eta_p\1_{(A_{p,1}(\mathfrak{s}_n))^c}])\\+L(\mathfrak{s}_n) \sum_{p\ge\lfloor(\log 1/\mathfrak{s}_n)^{1/2}\rfloor+1}\frac{1}{p^{1/2+\mathfrak{s}_n}}(\eta_p\1_{A_{p,1}(\mathfrak{s}_n)}-\mathbb{E}[\eta_p\1_{A_{p,1}(\mathfrak{s}_n)}])\\+L(\mathfrak{s}_n) \sum_p \frac{1}{p^{1/2+\mathfrak{s}_n}}(\eta_p\1_{(A_{p,1}(\mathfrak{s}_n))^c}-\mathbb{E}[\eta_p\1_{(A_{p,1}(\mathfrak{s}_n))^c}]).
        \end{multline*}
The first term on the right-hand side converges to $0$ a.s.\ as $n\to\infty$ by Lemma \ref{lemma_2}, with $M(s)=\lfloor(\log 1/s)^{1/2}\rfloor$. The second term does so according to formula \eqref{5.19} of Lemma \ref{lemma5.7}, with $N_1(s)=\lfloor(\log 1/s)^{1/2}\rfloor$. Finally, the third term vanishes by Lemma \ref{lemma_3}, again with $M(s)=\lfloor(\log 1/s)^{1/2}\rfloor$. By Lemma\footnote{Just in case, we attract the reader's attention to a somewhat delicate point. Although we apply Lemma \ref{lemma5.7}, with $M(s)=\lfloor(\log 1/s)^{1/2}\rfloor$, we are not supposed to take the same $M$ while proving Lemma \ref{lemma5.8}. Actually, $M$ used in the proof of Lemma \ref{lemma5.8} grows much faster than $s\mapsto \lfloor (\log 1/s)^{1/2}\rfloor$.} \ref{lemma5.8}, $$\limsup_{n\to\infty} L(\mathfrak{s}_n) \sum_p \frac{1}{p^{1/2+\mathfrak{s}_n}}(\eta_p\1_{(A_{p,1}(\mathfrak{s}_n))^c}-\mathbb{E}[\eta_p\1_{(A_{p,1}(\mathfrak{s}_n))^c}])\geq 1-\delta\quad\text{a.s.}$$ and thereupon
\begin{equation*}
\limsup_{s\to0+}L(s) \sum_p \frac{\eta_p}{p^{1/2+s}}\ge\limsup_{n\to\infty}L(\mathfrak{s}_n) \sum_p \frac{\eta_p}{p^{1/2+\mathfrak{s}_n}}\ge 1-\delta\quad\text{\rm a.s.}
\end{equation*}
Letting $\delta\to 0+$ completes the proof of \eqref{5.17}. Relation \eqref{5.18} follows from \eqref{5.17}, with $-\eta_p$ replacing $\eta_p$.
\end{proof}
\begin{proof}[Proof of Theorem \ref{con:lil}]
Relations \eqref{eq:limsup} and \eqref{eq:liminf} are secured by Propositions \ref{pr1} and \ref{pr2}, respectively.

It remains to prove \eqref{eq:limitpoints}. As a preparation, put $H:=\{z\in\mathbb{C}: {\rm Re}\,z>0\}$ and $$X(z):=\sum_p\frac{\eta_p}{p^{1/2+z}},\quad z\in H.$$ The so defined $X$ is a random analytic function. This implies that the random functions $s\to X(s)=\sum_p p^{-1/2-s}\eta_p$ and $s\mapsto (2\sigma^2\log 1/s\, \log^{(3)}1/s )^{-1/2}X(s)$ are a.s.\ continuous on $(0,\infty)$ and $(0,\eee^{-\eee})$, respectively. Now \eqref{eq:limitpoints} follows from \eqref{eq:limsup} and \eqref{eq:liminf} with the help of the intermediate value theorem for continuous functions.
\end{proof}

\section{Proof of Theorem \ref{con:clt}}\label{sect:flt}

We shall prove the result in an equivalent form
$$\Big(\frac{1}{s^{1/2}}\sum_p\frac{\eta_p}{p^{1/2+\exp(-ts)}}\Big)_{t\geq 0}~\Longrightarrow~ (\sigma B(t))_{t\geq 0},\quad s\to+\infty$$ on $C[0,\infty)$. As before, we can and do assume that $\sigma^2=1$.

We use a standard approach, which consists of two steps: (a) proving weak convergence of finite-dimensional distributions; (b) checking tightness.

\noindent (a) If $t=0$, then, for $s>0$, $X(\eee^{-ts})=X(1)=\sum_p p^{-3/2}\eta_p$, and $s^{-1/2}X(1)$ converges in probability to $B(0)=0$ as $s\to +\infty$.

Thus, it suffices to show that, for $t_1, t_2\in (0,\infty)$ (we do not need to consider $t=0$),
\begin{equation}\label{eq:covar}
\me \big[X(\eee^{-t_1s})X(\eee^{-t_2s})\big] ~\sim~ \min(t_1, t_2)s,\quad s\to +\infty
\end{equation}
and check the Lindeberg-Feller condition: for all $\varepsilon>0$ and each fixed $t>0$,
\begin{equation}\label{eq:Lindeberg}
\lim_{s\to +\infty}\frac{1}{s}\sum_p \me\Big[\Big(\frac{\eta_p}{p^{1/2+\exp(-ts)}}\Big)^2\1_{\{|\eta_p|>\varepsilon p^{1/2+\exp(-ts)}s^{1/2}\}}\Big]=0.
\end{equation}
\noindent {\sc Proof of \eqref{eq:covar}.} Using \eqref{eq:var} we obtain
\begin{multline*}
\me \Big[\sum_p \frac{\eta_p}{p^{1/2+\exp(-t_1s)}}\sum_{p_\ast}\frac{\eta_{p_\ast}}{p_\ast^{1/2+\exp(-t_2s)}}\Big]=\sum_p \frac{1}{p^{1+\exp(-t_1s)+\exp(-t_2s)}}\\~\sim~ -\log (\eee^{-t_1s}+\eee^{-t_2s})~\sim~ \min(t_1,t_2)s,\quad s\to +\infty.
\end{multline*}
{\sc Proof of \eqref{eq:Lindeberg}.} For each $p\in\mathcal{P}$, each $s>0$ and each $t>0$, $p^{-1/2-\exp(-ts)}\leq 1$. Hence, the expression under the limit on the left-hand side of \eqref{eq:Lindeberg} does not exceed $$\frac{\me [\eta^2\1_{\{|\eta|>\varepsilon s^{1/2}\}}]}{s}\sum_p\frac{1}{p^{1+2\exp(-ts)}}.$$  As shown in the proof of \eqref{eq:covar}, $\sum_p p^{-1-2\exp(-ts)} \sim ts$ as $s\to +\infty$. Further, $\me [\eta^2]<\infty$ entails $\lim_{s\to +\infty}\me [\eta^2\1_{\{|\eta|>\varepsilon s^{1/2}\}}]=0$. With these at hand, \eqref{eq:Lindeberg} follows.

\noindent (b) We have to prove tightness on $C[0,T]$ (the set of continuous functions defined on $[0,T]$) for each $T>0$. Since $(B(t))_{t\in [0,T]}$ has the same distribution as $T^{1/2}(B(t))_{t\in [0,1]}$, it is enough to investigate the case $T=1$ only.

As in the proof of Lemma \ref{lemma_2}, write, for $t\in [0,1]$ and $s\geq 1$,
\begin{multline*}
\Big|\sum_{p\leq \lfloor s\rfloor}\frac{\eta_p}{p^{1/2+\exp(-ts)}}\Big|=\Big|\frac{T^\ast(\lfloor s\rfloor)}{(\lfloor s\rfloor)^{1/2+\exp(-ts)}}+(1/2+\exp(-ts))\int_1^{\lfloor s 
\rfloor}\frac{T^\ast(x)}{x^{3/2+\exp(-ts)}}{\rm d}x\Big|\\\leq \frac{|T^\ast(\lfloor s\rfloor)|}{(\lfloor s\rfloor)^{1/2}}+ 3/2 \max_{y\leq \lfloor s\rfloor}\,|T^\ast(y)|\int_1^\infty \frac{{\rm d}x}{x^{3/2}}\quad\text{a.s.}
\end{multline*}
Donsker's functional limit theorem (see, for instance, Theorem 14.1 on p.~146 in \cite{Billingsley:1999}) entails $n^{-1/2}\max_{k\leq n}\,|T_k|\dod |\mathcal{N}(0,1)|$ as $n\to\infty$, where $\dod$ denotes convergence in distribution and $\mathcal{N}(0,1)$ denotes a random variable with the normal distribution of mean $0$ and variance $1$. This together with the prime number theorem ensures that
$$s^{-1/2}(\log s)^{1/2} \max_{y\leq \lfloor s\rfloor}\,|T^\ast(y)|=s^{-1/2}(\log s)^{1/2} \max_{k\leq \rho(\lfloor s\rfloor)}\,|T_k|~\dod~ |\mathcal{N}(0,1)|,\quad s\to\infty.$$ As a consequence,
\begin{equation}\label{eq:inter3}
\lim_{s\to +\infty} \frac{1}{s^{1/2}}\sup_{t\in [0,1]}\,\Big|\sum_{p\leq \lfloor s\rfloor}\frac{\eta_p}{p^{1/2+\exp(-ts)}}\Big|=0\quad\text{in probability}.
\end{equation}
Put $a(s):=(\log \lfloor s\rfloor)^{1/2}$ for $s\geq 1$. Arguing as in the proofs of \eqref{lem4.11} and \eqref{lem4.21} we infer
\begin{equation}\label{lem4.1112}
\lim_{s\to+\infty}\sup_{t\in [0,\,1]}\sum_{p\ge \lfloor s \rfloor+1}\frac{1}{p^{1/2+\exp(-ts)}}|\eta_p|\1_{\{|\eta_p|>(p/\log p)^{1/2}a(s)\}}=0\quad \text{\rm{a.s.}}
\end{equation}
and
\begin{equation}\label{lem4.2112}
\lim_{s\to\infty}\sup_{t\in [0,\,1]}\sum_{p\ge \lfloor s\rfloor+1}\frac{1}{p^{1/2+\exp(-ts)}}\mathbb{E}\big[|\eta_p|\1_{\{|\eta_p|>(p/\log p)^{1/2}a(s)\}}\big]=0.
\end{equation}
For instance, for $s\geq 1$, $$\sup_{t\in [0,1]}\sum_{p\ge \lfloor s\rfloor+1}\frac{1}{p^{1/2+\exp(-ts)}}|\eta_p|\1_{\{|\eta_p|>(p/\log p)^{1/2}a(s)\}}\leq \sum_{p\ge \lfloor s\rfloor+1}\frac{1}{p^{1/2}}|\eta_p|\1_{\{|\eta_p|>(p/\log p)^{1/2}a(s)\}}\quad\text{a.s.},$$ and the right-hand side is equal to $0$ for large enough $s$ a.s.

Put
\begin{equation*}
X^\ast(t,s):= \sum_{p\ge \lfloor s\rfloor +1}\frac{\eta^\ast_p(s)}{p^{1/2+\exp(-ts)}},\quad t\geq 0,~s>0,
\end{equation*}
where $\eta^\ast_p(s):=\eta_p\1_{\{|\eta_p|\leq (p/\log p)^{1/2}a(s)\}}-\me \big[\eta_p\1_{\{|\eta_p|\leq (p/\log p)^{1/2}a(s)\}}\big]$ for $p\in\mathcal{P}$ and $s>0$. Relations \eqref{eq:inter3}, \eqref{lem4.1112} and \eqref{lem4.2112} guarantee that we are left with proving tightness of the distributions of $(X^\ast(t,s))_{t\in [0,1]}$ for large $s>0$. By formula (7.8) on p.~82 in \cite{Billingsley:1999}, it is enough to show that, for all $\varepsilon>0$,
\begin{equation}\label{eq:tight}
\lim_{i\to\infty}\limsup_{s\to\infty}\mmp\{\sup_{u,v\in [0,1],\, |u-v|\leq 2^{-i}}\,|X^\ast(u,s)-X^\ast(v,s)|>\varepsilon s^{1/2}\}=0.
\end{equation}
The proof of \eqref{eq:tight} is analogous to the last part of the proof of Lemma \ref{lemma5}. We use dyadic partitions of $[0,1]$ by points $t^\ast_{j,\,m}:=2^{-j}m$ for $j\in\mn_0$ and $m=0,1,\ldots, 2^j$. Similarly to the argument preceding formula \eqref{5.12} we infer
\begin{equation*}
\sup_{u,v\in [0,1],\, |u-v|\leq 2^{-i}}\,|X^\ast(u,s)-X^\ast(v,s)|\leq \sum_{j\ge i}\max_{1\le m \le 2^j}|X^\ast(t^\ast_{j,\,m}, s)-X^\ast(t^\ast_{j,\,m-1},s)|.
\end{equation*}
Thus, it suffices to prove that, for all $\varepsilon>0$,
\begin{equation*}
\lim_{i\to\infty}\limsup_{s\to +\infty}\mathbb{P}\Big\{\sum_{j\ge i}\max_{1\le m \le 2^j}|X^\ast(t^\ast_{j,\,m},s)-X^\ast(t^\ast_{j,\,m-1},s)|>\varepsilon s^{1/2}\Big\}=0.
\end{equation*}
Put $a^\ast_j:=2^{-j/2}j^2$ for $j\in\mathbb{N}_0$. The last limit relation follows if we can show that, for all $\varepsilon>0$,
$$\lim_{i\to\infty}\limsup_{s\to +\infty}\sum_{j\ge i}\mathbb{P}\Big\{\max_{1\le m \le 2^j}|X^\ast(t^\ast_{j,\,m},s)-X^\ast(t^\ast_{j,\,m-1},s)|>\varepsilon a^\ast_j s^{1/2}\Big\}=0.$$
Denote by $A^\ast_j(s)$ and $B^\ast_j(u,p,s)$ the counterparts of $A_j(n)$ and $B_j(u,p,n)$, namely, for small $s>0$, $j\in\mn_0$ and $m\in [0,2^j]$,
$$A^\ast_j(s)=\sum_{p\geq  \lfloor s\rfloor+1}\frac{1}{p}\Big(\frac{1}{p^{\exp(-t^\ast_{j,\,m}s)}}-\frac{1}{p^{\exp(-t^\ast_{j,\,m-1}s)}}\Big)^2 
$$ and
$$B^\ast_j(u,p,s)=\frac{|u|}{p^{1/2}}\Big(\frac{1}{p^{\exp(-t^\ast_{j,\,m}s)}}-\frac{1}{p^{\exp(-t^\ast_{j,\,m-1}s)}}\Big)|\eta^\ast_p(s)|,\quad p\geq \lfloor s\rfloor+1.$$
Then $A^\ast_j(s)\leq C2^{-j}s$ and, if $2^j\geq s$, $B^\ast_j(u,p,s)\leq |u|2^{1-j}s=:C^\ast_j(u,s)$, if $2^j\leq s$, $B^\ast_j(u,p,s)\leq 2|u|=:C^\ast_j(u,s)$. With these at hand we obtain
\begin{equation*}
\me\big[\exp\big(\pm u(X^\ast(t^\ast_{j,\,m},s)-X^\ast(t^\ast_{j,\,m-1},s))\big)\big]\leq \exp\Big(\frac{C2^{-j}su^2}{2}\eee^{C^\ast_j(u,s)}\Big),\quad u\in\mathbb{R}
\end{equation*}
and thereupon
\begin{multline*}
\mathbb{P}\big\{|X^\ast(t^\ast_{j,\,m},s)-X^\ast(t^\ast_{j,\,m-1},s)|>\varepsilon a^\ast_j s^{1/2} \big\}\\ \leq \exp(-u \varepsilon a^\ast_j s^{1/2})\mathbb{E}[\exp(u|X^\ast(t^\ast_{j,\,m},s)-X^\ast(t^\ast_{j,\,m-1},s)|)]\\\leq  2\exp\Big(-u\varepsilon 2^{-j/2}j^2s^{1/2}+\frac{C2^{-j}su^2}{2}\eee^{C^\ast_j(u,s)}\Big).
\end{multline*}
Putting $u=\varepsilon 2^{j/2}s^{-1/2}$ we conclude that $C^\ast_j(u,s)\leq 2\varepsilon$ and further
\begin{multline*}
\sum_{j\ge i}\mathbb{P}\Big\{\max_{1\le m \le 2^j}|X^\ast(t^\ast_{j,\,m},s)-X^\ast(t^\ast_{j,\,m-1},s)|>\varepsilon a^\ast_j s^{1/2}\Big\}\\ \le 2\exp(C\varepsilon^2\eee^{2\varepsilon}/2)\sum_{j\ge i}2^j \eee^{-\varepsilon^2 j^2}~\to~ 0,\quad i\to\infty.
\end{multline*}
The proof of Theorem \ref{con:clt} is complete.

\section{Proofs of Corollaries \ref{cor:clt} and \ref{cor:lil}}\label{sect:cor}

Put $R(s):=\sum_p\sum_{k\geq 3}\frac{(-1)^{k+1}(f(p))^k}{kp^{k(1/2+s)}}$ for $s>0$. In the proof of Proposition 3.2 in \cite{Geis+Hiary:2025} (see also Lemma 2.4 in \cite{Aymone+Heap+Zhao:2023}) it is shown that $$\log \sum_{n\geq 1}\frac{f(n)}{n^{1/2+s}}+\frac{\log \zeta(1+2s)}{2}=\sum_p \frac{f(p)}{p^{1/2+s}}+R(s)+O(1)=\sum_p \frac{f(p)}{p^{1/2+s}}+ O(1),\quad s\to 0+\quad\text{a.s.}$$ Corollary \ref{cor:lil} follows from this relation and Theorem \ref{con:lil}.

Corollary \ref{cor:clt} follows from Theorem \ref{con:clt} and a relation that we are now going to prove: for all $T>0$, $$\lim_{s\to 0+}\frac{\sup_{t\in [0,\,T]}\,|R(s^t)|}{(\log 1/s)^{1/2}}=0\quad\text{in probability.}$$ Indeed, for $s\in (0,1)$ and $t\in [0,T]$, $$|R(s^t)|\leq \sum_p\sum_{k\geq 3}\frac{1}{kp^{k(1/2+s^t)}}\leq \sum_p \frac{1}{p^{3(1/2+s^t)}}\frac{p^{1/2+s^t}}{p^{1/2+s^t}-1}\leq \frac{2^{1/2}}{2^{1/2}-1}\sum_p \frac{1}{p^{3(1/2+s^T)}}.$$ The right-hand side converges to $\frac{2^{1/2}}{2^{1/2}-1}\sum_p \frac{1}{p^{3/2}}<\infty$ as $s\to 0+$. This justifies the claimed limit relation.

\section{Proof of Corollary \ref{cor:cltlil2}}\label{sect:cor3}

Starting with an Euler product representation $$\sum_{n\geq 1}\frac{f(n)}{n^{1/2+s}}=\prod_p \Big(1+\frac{f(p)}{p^{1/2+s}}+\frac{(f(p))^2}{p^{1+2s}}+\ldots+\frac{(f(p))^{k-1}}{p^{(1/2+s)(k-1)}}\Big),$$ write 
\begin{equation*}
\log\sum_{n\geq 1}\frac{f(n)}{n^{1/2+s}}= 
:\sum_p \frac{f(p)}{p^{1/2+s}}+\frac{1}{2}\sum_p \frac{1}{p^{1+2s}}+R^\ast(s). 
\end{equation*}
The fact that, for all $T>0$, $$\lim_{s\to 0+}\frac{\sup_{t\in [0,\,T]}\,|R^\ast(s^t)|}{(\log 1/s)^{1/2}}=0\quad\text{in probability}$$ follows along the lines of the previous proof.
With this at hand, invoking 
$$\log \sum_{n\geq 1}\frac{f(n)}{n^{1/2+s}}-\frac{\log \zeta(1+2s)}{2}=\sum_p \frac{f(p)}{p^{1/2+s}}+R^\ast(s)+O(1),\quad s\to 0+$$ completes the proof. 

\noindent \textbf{Acknowledgment.} A. Iksanov thanks Oleksiy Klurman for a useful correspondence and providing several relevant references. The authors are grateful to Paul Bourgade for his comment outlined in the paragraph following Theorem \ref{con:clt}.

\section{Statements \& Declarations}

\noindent \textbf{Funding.} This work was supported by the MOHRSS program (H 20240850), which is gratefully acknowledged.

\bigskip

\noindent \textbf{Competing Interests.} The authors have no relevant financial or non-financial interests to disclose.

\bigskip

\noindent \textbf{Data Availability.} Our manuscript has no associated data.

\bigskip

\noindent \textbf{Author Contributions.} Both authors contributed to the derivation of the results. Both authors contributed to the manuscript preparation, and both authors read and approved the final manuscript.

\end{document}